\documentclass[11pt,fleqn]{amsart}
\usepackage[paper=a4paper]
 {geometry}

\usepackage[small]{titlesec}

\usepackage{hyperref}

\usepackage{amsthm,thmtools}
\usepackage[utf8]{inputenc}
\usepackage[english]{babel}
\usepackage{enumerate}
\usepackage[osf,noBBpl]{mathpazo}
\usepackage[alphabetic,initials]{amsrefs}
\usepackage{amsfonts,amssymb,amsmath}
\usepackage{mathtools}
\usepackage{graphicx}
\usepackage[poly,arrow,curve,matrix]{xy}
\usepackage{helvet}

\renewcommand\thesection{\arabic{section}}
\titleformat{\section}
 {\normalfont\bfseries\large}
 {\thesection. \space}{0em}{}

\renewcommand\thesubsection{\arabic{subsection}}
\titleformat{\subsection}
 {\normalfont\bfseries}
 {\thesection.\thesubsection \space}{0em}{}

\renewcommand\proofname{Proof}

\makeatletter
\renewenvironment{proof}[1][\textit{\proofname}]{\par
 \pushQED{\qed}%
 \normalfont \topsep.75\paraskip\relax
 \trivlist
 \item[\hskip\labelsep
 \itshape
 #1\@addpunct{.}]\ignorespaces
}{%
 \popQED\endtrivlist\@endpefalse
}
\makeatother
\declaretheoremstyle[
bodyfont=\itshape,]{mystyle}
\declaretheorem[name=Lemma, style=mystyle, numberwithin=section]{Lemma}
\declaretheorem[name=Proposition, style=mystyle, sibling=Lemma]{Proposition}
\declaretheorem[name=Theorem, style=mystyle, sibling=Lemma]{Theorem}

\declaretheorem[name=Definition, style=mystyle, sibling=Lemma]{Definition}

\declaretheoremstyle[numbered=no, 
bodyfont=\itshape]{mystyle-empty}
\declaretheorem[name=Lemma, style=mystyle-empty]{Lemma*}
\declaretheorem[name=Proposition, style=mystyle-empty]{Proposition*}
\declaretheorem[name=Theorem, style=mystyle-empty]{Theorem*}
\declaretheorem[name=Corollary, style=mystyle-empty]{Corollary*}
\declaretheorem[name=Definition, style=mystyle-empty]{Definition*}
\declaretheorem[name=Example, style=mystyle-empty]{Example*}
\declaretheorem[name=Remark, style=mystyle-empty]{Remark*}

\newskip\paraskip
\newcounter{para}[section]
\renewcommand\thepara{\thesection.\arabic{para}}
\def\paragraph{%
 \noindent
 \refstepcounter{para}%
 \textbf{\thepara.}\hspace{1ex}%
}

\newcommand\about[1]{%
 {\bfseries#1.}%
}

\newcommand\NN{\mathbb N}
\newcommand\CC{\mathbb C}

\newcommand\RR{\mathbb R}
\newcommand\ZZ{\mathbb Z}
\newcommand\II{\mathbb I}

\newcommand\ot{\otimes}
\renewcommand\to{\longrightarrow}
\renewcommand\phi{\varphi}
\newcommand\Id{\mathsf{Id}}

\newcommand\A{\mathcal A}

\newcommand\D[3]{{}^{#1} \mathfrak D_{#2}^{#3}}
\newcommand\DD[3]{{}^{#1} \mathcal D_{#2}^{#3}}
\newcommand\Z{\mathsf Z}

\newcommand\m{\mathfrak m}
\newcommand\gl{\mathfrak{gl}}

\renewcommand\SS{\mathfrak S}
\newcommand\vv{\overline{v}}

\newcommand\vectspan[1]{\left\langle #1 \right\rangle}

\DeclareMathOperator\Frac{Frac}
\DeclareMathOperator\Specm{Specm}
\DeclareMathOperator\End{End}
\DeclareMathOperator\Hom{Hom}

\DeclareMathOperator\sym{sym}

\DeclareMathOperator\sg{sg}
\DeclareMathOperator\ev{\mathsf{ev}}
\DeclareMathOperator\ann{ann}
\DeclareMathOperator\supp{supp}

\begin{document}
\title{Gelfand-Tsetlin Theory for Rational Galois Algebras}
\date{}

\author[V.Futorny]{Vyacheslav Futorny}
\address{Instituto de Matem\'atica e Estat\'istica, Universidade de S\~ao
Paulo, S\~ao Paulo SP, Brasil} \email{futorny@ime.usp.br,}
\author[D.Gratcharov]{Dimitar Grantcharov}
\address{\noindent
University of Texas at Arlington, Arlington, TX 76019, USA} \email{grandim@uta.edu}
\author[L.E.Ramirez]{Luis Enrique Ramirez}
\address{Universidade Federal do ABC, Santo Andr\'e-SP, Brasil} \email{luis.enrique@ufabc.edu.br,}
\author[P.Zadunaisky]{Pablo Zadunaisky}
\address{Instituto de Matem\'atica e Estat\'istica, Universidade de S\~ao
Paulo, S\~ao Paulo SP, Brasil} \email{pzadun@ime.usp.br}

\begin{abstract} 
In the present paper we study Gelfand-Tsetlin modules defined in terms of BGG differential operators. The structure of these modules is described with the aid of the Postnikov-Stanley polynomials introduced in \cite{PS-chains-bruhat}. These polynomials are used to identify the  action of the Gelfand-Tsetlin subalgebra on the BGG operators. We also provide explicit bases of the corresponding Gelfand-Tsetlin modules and prove a simplicity criterion for these modules.  The results hold for modules  defined over standard Galois orders of type $A$ - a large class of rings  that include the universal enveloping algebra of $\mathfrak{gl} (n)$ and the finite $W$-algebras of type $A$.
\end{abstract}
\maketitle


\section{Introduction}

The category of Gelfand-Tsetlin modules of the general linear Lie algebra 
$\mathfrak{gl} (n)$ is an important category of modules that plays a prominent 
role in many areas of mathematics and theoretical physics. By definition, a 
Gelfand-Tsetlin module of $\mathfrak{gl} (n)$ is one that has a generalized 
eigenspace decomposition over a certain maximal commutative subalgebra 
(\emph{Gelfand-Tsetlin subalgebra}) $\Gamma$ of the universal enveloping 
algebra of $\mathfrak{gl} (n)$. This algebraic definition has a nice 
combinatorial flavor. The concept of a Gelfand-Tsetlin module generalizes the 
classical realization of the simple finite-dimensional representations of 
$\mathfrak{gl} (n)$ via the so-called Gelfand-Tsetlin tableaux introduced in 
\cite{GT}. The explicit nature of the Gelfand-Tsetlin formulas inevitably 
raises the question of what infinite-dimensional modules admit tableaux bases 
- a question that led to the systematic study of  the theory of 
Gelfand-Tsetlin modules. This theory has attracted considerable attention in 
the last 30 years of the 20th century and have been studied in \cites{
DFO-GT-modules-original, DFO-GT-modules, Maz1, Maz2, m:gtsb, Zh}, among 
others. Gelfand-Tsetlin bases and modules are also related to  Gelfand-Tsetlin 
integrable systems that were first introduced for the unitary Lie algebra 
${\mathfrak u}(n)$ by Guillemin and Sternberg in \cite{GS}, and later for the 
general linear Lie algebra $\mathfrak{gl}(n)$ by Kostant and Wallach in 
\cite{KW-1} and \cite{KW-2}.

Recently, the study of Gelfand-Tsetlin modules took a new direction after the theory of singular Gelfand-Tsetlin modules was initiated in \cite{FGR-1-singular}. Singular Gelfand-Tsetlin modules are roughly those that have basis of tableaux whose  entries may be zeros of the denominators in the Gelfand-Tsetlin formulas. For the last three years remarkable progress has been made towards the study of singular Gelfand-Tsetlin modules of $\mathfrak{gl}(n)$. Important results in this direction were obtained in \cites{FGR-generic-irreducible, FGR-1-singular, FGR-2-index, Zad-1-sing, 
Vis-geometric-1-singular-GT, Vis-geometric-singular-GT, 
RZ-singular-characters}. In particular, explicit constructions of a Gelfand-Tsetlin module with a fixed singular Gelfand-Tsetlin character were obtained with algebro-combinatorial methods in \cite{RZ-singular-characters} and with geometric methods in \cite{Vis-geometric-singular-GT}.  One notable property of these general constructions is their relations with Schubert calculus and reflection groups. As explained below, this relation  is brought to a higher level in the present paper and new connections with Schubert polynomials and generalized Littlewood-Richardson coefficients are established. We hope that these new connections, combined with combinatorial results on skew Schubert polynomials, will help us to bring within a reach the solution of the most important problem in the theory: the classification of all simple Gelfand-Tsetlin modules of $\mathfrak{gl} (n)$.

The study of Gelfand-Tsetlin modules is not limited to the cases of $\mathfrak{gl} (n)$ and $\mathfrak{sl} (n)$.  Gelfand-Tsetlin subalgebras are part of a uniform algebraic theory, the theory of Galois orders. Galois orders are special types of rings that were introduced in 
\cite{FO-galois-orders} in an attempt to unify the  representation theories of 
generalized Weyl algebras and the universal 
enveloping algebra of $\gl (n)$. In addition to 
the universal enveloping algebra of
$\gl (n)$ examples of Galois orders include  the $n$-th Weyl algebra,  the quantum plane,  the Witten-Wo\-ro\-no\-wicz algebra, the $q$-deformed Heisenberg algebra, and finite
$W$-algebras of type $A$ (for details and more examples see for example \cite{Hart-rational-galois}).

The representation theory of Galois orders was initiated in \cite{FO-fibers-gt}. In particular, the following finiteness theorem for Gelfand-Tsetlin modules of a Galois order $U$ over an integral domain $\Gamma$ was proven: given a maximal ideal $\m$ of $\Gamma$ there exists only finitely many non-isomorphic simple Gelfand-Tsetlin modules $M$ such that $M[\m] \neq 0$ (see \S \ref{par-gt-modules} for the definition of $M[\m]$). This theorem generalizes the finiteness theorem for $\gl (n)$ obtained in \cite{Ovs-finiteness}. Other important results of the Gelfand-Tsetlin theory of $\gl (n)$ were extended to  certain types of Galois orders in \cites{EMV-orthogonal, Hart-rational-galois, Maz-orthogonal-GT-alg}. One such important result is the construction of a Gelfand-Tsetlin module with any fixed Gelfand-Tsetlin character over an orthogonal Gelfand-Tsetlin algebra obtained very recently in \cite{EMV-orthogonal}. Another notable contribution is the new framework of rational Galois orders established in \cite{Hart-rational-galois}. Examples of rational Galois orders are the universal enveloping algebra of $\gl (n)$, restricted Yangians of $\gl (n)$, orthogonal Gelfand-Tsetlin algebras, finite $W$-algebras of type $A$, among  others.

The first goal of the present paper is to establish a closer connection of the singular Gelfand-Tstelin theory with the theory of  Schubert polynomials and reflection groups.  We study a new natural class of $\Gamma$-modules that consists of differential operators related to the polynomials introduced  in \cite{BGG-cohomology}. These BGG differential operators have numerous applications in the cohomology theory of flag varieties. In the present paper, we use a particular aspect of these applications - the Postnikov-Stanley operators. Postnikov-Stanley polynomials were originally defined in \cite{PS-chains-bruhat} in order to  express degrees of Schubert varieties in the generalized complex flag manifold $G/B$. The polynomials are given by weighted sums over saturated chains in the Bruhat order and have intimate relations with Schubert polynomials, harmonic polynomials, Demazure characters, and generalized Littlewood-Richardson coefficients. The action of $\Gamma$ on the module of BGG differential operators is described explicitly in terms of Postnikov-Stanley operators. Using this explicit action, we prove one of our main results -  an upper bound for the size of Jordan blocks of the generators of $\Gamma$, see Theorem \ref{thm-Jordan}. The explicitness of the action  helps us also to understand better the structure of the $U$-module consisting of BGG operators and, in particular, is used as a technical tool in the proof of the simplicity criterion for this $U$-module

The other goal of the paper is to deepen the study of Gelfand-Tsetlin  modules of  rational Galois orders. We define a special class of rational Galois orders that we call standard Galois orders of type $A$ that includes most of the examples of rational Galois orders listed above. Then we construct Gelfand-Tsetlin modules of arbitrary character over these Galois orders and provide explicit bases of these modules, see Theorem \ref{T:direct-sum}. Our last main result, Theorem \ref{T:simplicity}, is a sufficient condition for these modules to be simple. This simplicity criterion 
generalizes the criterion for orthogonal Gelfand-Tsetlin algebras obtained in \cite{EMV-orthogonal}. It is worth noting,  that as an immediate corollary, our result provides new examples of  simple modules of any finite $W$-algebra of type $A$.

The organization of the paper is as follows. Preliminary results on reflection groups are collected in Section 2. In Section 3 we include the needed background on BGG differential operators and Postnikov-Stanley differential operators. Definitions and properties of Galois orders and Gelfand-Tsetlin modules are included in Section 4. In Section 5 we discuss generalities on rational Galois orders. The  $\Gamma$-module of BGG operators is defined in Section 6, where we study its structure with the aid of Postnikov-Stanley  operators. In this section we also give an upper bound for the size of a Jordan block of any $\gamma$ of $\Gamma$ considered as an endomorphism of the $\Gamma$-module of BGG differential operators. The $U$-module structure of a (larger) space of BGG differential operators is studied in  Section 7. In Section 8 we provide a basis of the $U$-module defined in Section 7, prove that this module is a Gelfand-Tsetlin module, and establish a simplicity criterion for this module.

We finish the introduction with a few notational conventions, which will be 
used throughout the paper. Unless otherwise stated, the ground field will be 
$\mathbb C$. By $\mathbb N$ we denote the set of positive integer numbers. 
A \emph{reflection group} will always be a finite group isomorphic to a 
subgroup of $\mathsf O(n, \RR)$ for some $n \in \NN$ and generated by reflections.
Given a ring $R$ and a monoid $\mathcal M$ acting on $R$ by ring morphisms,
 by $R \# \mathcal M$ we denote the \emph{smash product} of $R$ and $\mathcal 
M$, i.e. the free $R$-module with basis $\mathcal M$ and product given by
$r_1 m_1 \cdot r_2 m_2 = r_1 m_1(r_2) m_1m_2$ for any $r_1, r_2 \in R$ and
any $m_1, m_2 \in \mathcal M$.

\noindent{\bf Acknowledgements.} V.F. is
supported in part by CNPq grant (301320/2013-6) and by 
Fapesp grant (2014/09310-5). D.G is supported in part by Simons Collaboration 
Grant 358245. P.Z. is supported by Fapesp fellowship (2016-25984-1). 

\section{Preliminaries on reflection groups}
We recall some basic facts and fix notation on root systems and reflection 
groups. Our definition of root system is slightly different from the classical
one, but is easily seen to be equivalent.

\paragraph
\about{Root systems and reflection groups} Let $V$ be a complex vector space 
with a fixed inner product which we denote by $(-,-)$. We use this inner 
product to identify $V$ with its dual $V^*$ and for 
each $\alpha \in V^*$ we denote by $v_\alpha$ the unique element of $V$ such 
that $\alpha(v') = (v', v_\alpha)$ for all $v' \in V$. Given $\alpha \in V^*$ 
we denote by 
$s_\alpha$ the orthogonal reflection through the hyperplane $\ker \alpha$, and 
by $s_\alpha^*$ the corresponding endomorphism of $V^*$. In this article a
\emph{finite root system} over $V$ will be a finite set $\Phi \subset V^*$ 
such that for each $\alpha \in \Phi$ we have 
\begin{itemize}
\item[(R1)] $\Phi \cap \CC \alpha = \{\pm \alpha\}$ and
\item[(R2)] $s_\alpha^*(\Phi) \subset \Phi$.
\end{itemize}
In classical references such as \cite{Hump-coxeter-book} and 
\cite{Hiller-coxeter-book} root systems are defined as subsets of an Euclidian
vector space $V_\RR$ with $\RR$ instead of $\CC$ in $(R1)$. Taking $V = \CC 
\ot_\RR V_\RR$ for an adequate $V_\RR$ our definition is equivalent to theirs. 
We use the definition above since we work with complex vector spaces endowed 
with the action of a reflection group. 

We now review the basic features of the theory of root systems. For more 
details we refer the reader to the two references above. Fix a root system 
$\Phi$. The Weyl group associated to $\Phi$ is the group $W(\Phi)$ generated 
by $\{s_\alpha \mid \alpha \in \Phi\}$. Since we do not assume that the root 
systems are reduced or crystallographic, nor that $\Phi$ generates $V_\RR^*$,  
the group $W(\Phi)$ is a finite reflection group which may be 
decomposable, and its action on $V$ may have a nontrivial stabilizer. Any 
reflection group $G \subset \mathsf{GL}(V)$ is the Weyl group of some root 
system $\Phi \subset V^*$ \cite{Hiller-coxeter-book}*{\S 1.2}. 

Just as in the case of root systems for Lie algebras, for each root system 
$\Phi$ we can choose a linearly independent subset $\Sigma \subset \Phi$ which
is a basis of the $\RR$-span of $\Phi$ such that the coefficients of each root
of $\Phi$ in this basis are either all nonnegative or all nonpositive. Such 
sets are called \emph{bases} or \emph{simple systems}, and its elements are 
called \emph{simple roots}. Each choice of a base defines a partition $\Phi = 
\Phi^+ \cup - \Phi^+$, where $\Phi^+$ is the set of all \emph{positive} roots, 
i.e. those whose coordinates over $\Sigma$ are nonnegative. If we fix a base 
$\Sigma$ then the set $S$ of reflections corresponding to simple roots is a 
minimal generating set of the reflection group $W = W(\Phi)$, and hence $(W,S)$
is a finite Coxeter system in the sense of \cite{Hump-coxeter-book}*{1.9}. 
Each $s \in W$ of order two is of the form $s_\alpha$ for some $\alpha \in 
\Phi^+$ \cite{Hump-coxeter-book}*{Proposition 2.14}, and given $s \in W$ of 
order two we denote by $\alpha_s$ the corresponding positive root.

Fixing a base $\Sigma$, or equivalently, a minimal generating set $S \subset 
W$, we define the \emph{length} $\ell(\sigma)$ of $\sigma \in W$ as the least 
positive integer $\ell$ such that $\sigma$ can be written as a composition of 
$\ell$ reflections in $S$. Any sequence $s_1, \ldots, s_{\ell(\sigma)}$ such 
that $\sigma = s_1 \cdots s_{\ell(\sigma)}$ is called a \emph{reduced 
decomposition}; notice that reduced decompositions are not unique. The group 
$W$ acts faithfully and transitively on 
$\Phi$. Furthermore, $\ell(\sigma) = |\sigma(\Phi^+) \cap -\Phi^+|$, so $W$ 
has a unique longest element whose length equals $|\Phi|$. We will denote this 
element by $\omega_0(W)$, or simply by $\omega_0$ if the group $W$ is clear 
from the context.

For the rest of this section we fix a root system $\Phi$ with base $\Sigma$ 
and denote by $(W,S)$ be the corresponding Coxeter system. 

\paragraph
\about{Subsystems, subgroups and stabilizers} In this subsection we follow
\cite{Hump-coxeter-book}*{1.10}, where the reader can find most proofs.
Given $\Omega \subset \Sigma$ we denote by 
$\Phi(\Omega)$ the root subsystem generated by $\Omega$. We will call such 
subsystems \emph{standard}. If $\Psi \subset \Phi$ is an arbitrary subsystem 
then we can choose a base $\Omega \subset \Psi$ which can be extended to a 
base $\overline \Omega$ of $\Phi$. By \cite{Hump-coxeter-book}*{1.4 Theorem} 
$W$ acts transitively on the set of all bases of $\Phi$, so for some $\sigma 
\in W$ we have $\sigma(\overline \Omega) = \Sigma$ and hence $\sigma(\Psi)$ is 
standard. 

Let $\theta \subset S$ and denote by $W_\theta$ the subgroup of $W$ generated 
by $\theta$. Then $(W_\theta, \theta)$ is also a Coxeter system and it 
determines a standard root system $\Phi_\theta \subset \Phi$ with simple roots 
$\Sigma_\theta = \{\alpha_s \mid s \in \theta\}$. We will refer to subgroups 
of the form $W_\theta$ as \emph{standard parabolic} subgroups. A parabolic 
subgroup is any subgroup of $W$ that is conjugate to a standard parabolic 
subgroup.

If $\sigma \in W_\theta$ then we can compute its length as an element of $W$ 
with respect to the generating set $S$ or as an element of $W_\theta$ with 
respect to the generating set $\theta$. Both lengths turn out to be 
equal and will be denoted by $\ell(\sigma)$. Since $W_\theta$ is also a 
Coxeter group it has a unique element of maximal length which we will denote 
by $\omega_0(\theta)$. The set $W^\theta = \{\sigma \in W \mid \ell(\sigma s) >
\ell(\sigma) \mbox{ for all } s \in \theta\}$ is a set of representatives of 
the classes in the quotient $W/W_\theta$, and for each $\sigma \in W$ there 
exist unique elements $\sigma^\theta \in W^\theta$ and $\sigma_\theta \in 
W_\theta$ such that $\sigma = \sigma^\theta\sigma_\theta$ with $\ell(\sigma) = 
\ell(\sigma^\theta) + \ell(\sigma_\theta)$. The element $\sigma^\theta$ is the 
element of minimal length in the coclass $\sigma W_\theta$. It follows that 
$(\omega_0)_\theta = \omega_0(\theta)$ and therefore $\omega_0^\theta = 
\omega_0 \omega_0(\theta)^{-1}$.

Given $v \in V$ we denote by $\Phi_0(v)$ the set of all roots in $\Phi$ such 
that $\alpha(v) = 0$, which is clearly a root subsystem of $\Phi$. We also 
denote by $W_v$ the stabilizer of $v$ in $W$. We will say that $v$ is 
\emph{$\Sigma$-standard}, or just \emph{standard} when $\Sigma$ is fixed or 
clear from the context, if $\Phi_0(v)$ is a $\Sigma$-standard subsystem of 
$\Phi$. It is easy to check that $v$ is standard if and only if $W_v$ is a 
standard parabolic subgroup, and $W =W(\Phi_0(v))$. Since $W_{\sigma(v)} = 
\sigma W_v \sigma^{-1}$ and $\Phi_0(\sigma(v)) = \sigma(\Phi_0(v))$ for all 
$\sigma \in W$, it follows that for every $v \in V$ there exists $\sigma \in 
W$ such that $\sigma(v)$ is standard and hence $W_{\sigma(v)}$ is a standard 
parabolic subgroup. If $v$ is standard then we denote by $W^v$ the set of 
minimal length representatives of the left coclasses $W/W_v$.

\section{Divided differences and Postnikov-Stanley operators}

In this section $V$ is a fixed complex vector space, $\Lambda = S(V)$, and $L$ is 
the fraction field of $\Lambda$.  Note that following the convention of 
\cite{PS-chains-bruhat}, we  write $S(V)$ for $Sym(V^*)$. Also, we fix a finite root system $\Phi$ with 
base $\Sigma$, and set $W = W(\Phi)$ to be the corresponding reflection group 
with minimal generating set $S$. Thus $W$ acts on $\Lambda$ and $L$, and we 
set $\Gamma = \Lambda^W$ and $K = L^W$.

\paragraph
\about{Divided differences}
Since $W$ acts on $L$ we can form the smash product $L \# W$. Recall that the 
product in this complex algebra is given over generators by $f \sigma \cdot g 
\tau = f \sigma(g) \sigma \tau$ for all $f, g \in L$ and all $\sigma, \tau 
\in W$. Dedekind's theorem on linear independence of field homomorphisms 
implies that the algebra morhpism $L \# W \hookrightarrow \End_\CC(L)$ defined 
by mapping $l \sigma \in L \# W$ to the endomorphism $f \mapsto l\sigma(f)$ is
an embedding. We identify $L\# W$ with its image, and so must be careful to 
distinguish the result of applying the endomorphism $l \sigma$ to $f$, whose
result is $l\sigma(f)$, and the product of $l \sigma$ and $f$ in $L \# W$, 
which is $l \sigma \cdot f = l\sigma(f) \sigma$. 

For $s\in W$ we set
\begin{align*}
\nabla_s 
	&= \frac{1}{\alpha_s}(1-s) \in L \# W.
\end{align*}
It is easy to show that for each $f, g \in L$,
\begin{align*}
\nabla_s(fg) = \nabla_s(f) g + s(f)\nabla_s(g)
\end{align*}
so $\nabla_s$ is a twisted 
derivation of $L$. Notice that $\ker \nabla_s$ is exactly $L^{\vectspan s}$ 
and so $\nabla_s$ is $L^{\vectspan s}$-linear. Also it follows from the
definition that $\nabla_s(\Lambda) \subset \Lambda$. 

\begin{Example*}
Suppose $V = \CC^2$ and let $\{x,y\} \subset (\CC^2)^*$ be the dual basis
to the canonical basis. Let $s$ be the reflection given by $s(z_1, z_2) = (z_2,
z_1)$, so $\alpha_s = x -y$. Then for each $f(x,y) \in \CC[x,y]$ we have
$\nabla_s (f)(x,y) = \frac{f(x,y) - f(y,x)}{x-y}$. Notice that this quotient is
always a polynomial, since $f(x,y) - f(y,x)$ is an antisymmetric polynomial and
hence divisible by $x-y$.
\end{Example*}

Given $\sigma \in W$ we take a reduced decomposition $\sigma = s_1 \cdots 
s_{\ell}$ and set $\partial_\sigma = \nabla_{s_1} \circ \cdots \circ 
\nabla_{s_\ell}$; this element is called the \emph{divided difference} 
corresponding to $\sigma$ and does not depend on the chosen reduced 
decomposition \cite{Hiller-coxeter-book}*{Chapter IV (1.6)}. Notice though that
the definition of $\partial_\sigma$ does depend on the choice of a base 
$\Sigma \subset \Phi$. 

By definition, an $L\# W$-module $Z$ is an $L$-vector space 
endowed with a $W$-module structure such that the action of $L$ 
on $Z$ is $W$-equivariant. A simple induction on the length of $\sigma$ shows 
that the divided difference $\partial_\sigma$ defines a $K$-linear map over 
any $L\# W$-module $Z$. In particular $L$ is such a module, and since 
$\nabla_s(\Lambda) \subset \Lambda$ for any $s \in S$, it follows that 
$\Lambda$ is closed under the action of divided differences.

\paragraph
\about{Coinvariant spaces and Schubert polynomials}
The algebra $\Lambda$ is  $\ZZ_{\ge 0}$-graded with $\Lambda_1 = V^*$
and $\Gamma$ is a graded subalgebra of $\Lambda$ . We denote by $I_W$ the 
ideal of $\Lambda$ generated by the elements of $\Gamma$ of positive degree. 
By the Chevalley-Shephard-Todd theorem $\Gamma$ is isomorphic to a polynomial 
algebra in $\dim V$ variables and $\Lambda$ is a free $\Gamma$-module of rank 
$|W|$. Also, a set $B \subset \Lambda$ is a basis of the $\Gamma$-module 
$\Lambda$ if and only if its image in the quotient $\Lambda/I_W$ is a 
$\CC$-basis. Furthermore, $\Lambda/I_W$ is naturally a graded $W$-module 
isomorphic to the regular representation of $W$ with Hilbert series 
$\sum_{\sigma \in W} t^{\ell(\sigma)}$. For proofs we refer the reader to 
\cite{Hiller-coxeter-book}*{Chapter II, Section 3}.

We now recall the construction of the basis of Schubert polynomials of $\Lambda
/I_W$. This construction is due to Bernstein, Gelfand and Gelfand 
\cite{BGG-cohomology} and Demazure \cite{Dem-schubert} in the case when $W$ is 
a Weyl group, and to Hiller \cite{Hiller-coxeter-book}*{Chapter IV} in the 
case of arbitrary Coxeter groups. Set $\Delta(\Phi) = 
\prod_{\alpha \in \Phi^+} \alpha$, and for each 
$\sigma \in W$ set $\SS_\sigma^\Sigma = \frac{1}{|W|} \partial_{\sigma^{-1}
\omega_0} \Delta(\Phi)$. We will often write $\SS_\sigma$ instead of 
$\SS_\sigma^\Sigma$ when the base $\Sigma$ is clear from the context. Notice 
that by definition $\deg \SS_\sigma 
= \ell(\sigma)$. The polynomials $\{\SS_\sigma \mid \sigma \in W\}$ are known
as \emph{Schubert polynomials}, and they form a basis of 
$\Lambda$ as a $\Gamma$-module, so the projection of this set is a basis 
of $\Lambda/I_W$ as a complex vector space. Since $K = L^W$ we know that $L$ 
is a $K$-vector space of dimension $|W|$ and so $\{\SS_\sigma \mid \sigma \in 
W\}$ is also a basis of $L$ over $K$. Given $f \in L$ we will denote by 
$f_{(\sigma)}$ the coefficient of $\SS_\sigma$ in the expansion of $f$ 
relative to this basis, so $f = \sum_{\sigma \in W} f_{(\sigma)} \SS_\sigma$. 

Since Schubert polynomials form a basis of $\Lambda / I_W$, for all $\sigma,
\tau, \rho \in W$ there exists $c_{\sigma, \tau}^\rho \in \CC$ defined 
implicitly by the equation
\begin{align*}
\SS_\sigma \SS_\tau 
	&= \sum_{\rho \in W} c^\rho_{\sigma, \tau} \SS_\rho \mod I_W.
\end{align*}
The coefficients $c_{\sigma, \tau}^\rho$ are the \emph{generalized 
Littlewood-Richardson coefficients} relative to the base $\Sigma$. It follows 
from the definition that $c^\rho_{\sigma, \tau} = 0$ unless $\ell(\sigma) + 
\ell(\tau) = \ell(\rho)$. If $\theta \subset S$ then the space of 
$W_\theta$-invariants $(\Lambda/I_W)^{W_\theta}$ is generated by the set 
$\{\SS_\sigma \mid \sigma \in W^\theta\}$ 
\cite{Hiller-coxeter-book}*{Chapter IV (4.4)}. In particular, if $\sigma, 
\tau \in W^\theta$ then $c^{\rho}_{\sigma, \tau} \neq 0$ implies that 
$\rho \in W^\theta$. 

\paragraph
\about{Postnikov-Stanley operators}
\label{ps-operators}
Given $\alpha \in V^*$ there is a unique $\CC$-linear derivation 
$\Theta(\alpha): \Lambda \to \Lambda$ such that $\Theta(\alpha)(\beta) = 
(\beta, \alpha)$ for each $\beta \in V^*$. This map extends uniquely to a 
morphism $\Theta: \Lambda \to \mathsf{Der}_\CC(\Lambda)$. If we fix an 
orthonormal basis $x_1, \ldots, x_n$ of $V^*$, then $S(V) \cong \CC[x_1, 
\ldots, x_n]$ and $\Theta(x_i) = \frac{\partial}{\partial x_i}$.

Let $(-,-)_\Theta: \Lambda \times \Lambda \to \CC$ be the bilinear form given 
by $(f,g) = \Theta(f)(g)(0)$. This is a nondegenerate bilinear form which can 
be used to identify $\Lambda$ with its graded dual $\Lambda^\circ$. For every graded ideal $I 
\subset \Lambda$ we write $\mathcal H_I = \{g \in \Lambda \mid (f,g)_\Theta = 0
\mbox{ for all } f \in I\}$. Since the pairing $(-,-)_\Theta$ is nodegenerate, 
the space $\mathcal H_I$ is naturally isomorphic to the graded dual 
$(\Lambda/I)^\circ$. We denote by $P_\sigma^\Sigma$ the unique element in 
$\mathcal H_{I_W}$ such that $(P_\sigma^\Sigma, \SS_{\tau}^\Sigma) = 
\delta_{\sigma, \tau}$ for all $\sigma, \tau \in W$. Like before, we usually 
write $P_\sigma$ instead of $P_\sigma^\Sigma$. It follows that the set 
$\{P_\sigma \mid \sigma \in W\}$ is a graded basis of $\mathcal H_{I_W}$, dual 
to the Demazure basis of $\Lambda/I_W$. Also for each $\theta \subset S$ the 
set $\{P_\sigma \mid \sigma \in W^\theta\}$ is a graded basis of the dual of 
$(\Lambda/I_W)^{W_\theta}$. Notice that both these families are bases of
the space of $W$-harmonic polynomials, i.e. those polynomials which are 
annihilated by $W$-symmetric differential operators.

Recall that $\sigma$ \emph{covers} $\tau$, and denote this by $\tau
\preceq \sigma$, if $\sigma = \tau s_\alpha$ for some $\alpha \in \Phi$ and 
$\ell(\sigma) = \ell(\tau) + 1$. The \emph{Bruhat order} of $W$ is the 
transitive closure of this relation. A \emph{saturated chain} from $\sigma$ to 
$\tau$ in the Bruhat order is a sequence $\sigma = \sigma_0 \preceq \sigma_1 
\preceq \cdots \preceq \sigma_r = \tau$, and we refer to $r$ as the length of 
the saturated chain. The polynomials $P_\sigma$ were described by Postnikov 
and Stanley in terms of saturated chains in the Bruhat order of $W$ in 
\cite{PS-chains-bruhat} when $W$ is a Weyl group. 

For each covering relation $\sigma \preceq \sigma s_\alpha$ with 
$\alpha \in \Phi^+$ we set $m(\sigma,\sigma s_\alpha) = \alpha \in V^* = 
S(V)_1$, and for a saturated chain $C = (\sigma_1, \sigma_2, \ldots, 
\sigma_r)$ we denote by $m_C$ the product $\prod_{i=1}^{r-1} 
m_C(\sigma_i,\sigma_{i+1})$. Set
\begin{align*}
P_{\sigma, \tau} &= \frac{1}{(\ell(\tau) - \ell(\sigma))!}\sum_C m_C
\end{align*}
where the sum is taken over all saturated chains from $\sigma$ to $\tau$. 
Now, according to \cite{PS-chains-bruhat}*{Corollary 6.9}, if $\sigma \leq 
\tau$ in the Bruhat order then $P_{\sigma,\tau} = \sum_{\rho \in W} 
c^{\tau}_{\sigma,\rho} P_{\rho}$. This identity inspires the following 
definition.
\begin{Definition}
For each $\sigma, \tau \in W$ with $\tau \leq \sigma$ in the Bruhat order of 
$W$ we set $\D{\Sigma}{\sigma}{} = \Theta(P_\sigma^\Sigma)$ and 
$\D{\Sigma}{\tau,\sigma}{} = \sum_{\rho \in W} c^\sigma_{\tau, \rho} 
\D{\Sigma}{\rho}{}$. We ommit the superscript $\Sigma$ whenever it is clear 
from the context.
\end{Definition}

\paragraph
Notice that although by definition $\D{\Sigma}{\tau, \sigma}{}$ is a 
differential operator on $\Lambda$, it has a well defined extension to the 
fraction field $L$, and we will denote this extenssion by the same symbol. We 
denote by $\D{}{\sigma}{0}$ and $\D{}{\tau, \sigma}{0}$ the linear functional 
of $\Lambda$ obtained by applying the corresponding differential operator 
followed by evaluation at $0$. The definition of the polynomials $P_\sigma$ 
implies that $\D{}{\sigma}{0}(\gamma f) = \gamma(0) \D{}{\sigma}{0}(f)$ for 
all $f \in \Lambda$ and $\gamma \in \Gamma$. The following proposition shows 
that this functional extends to the algebra of rational functions without 
poles at $0$ and gives a generalized Leibniz rule to compute the result of 
applying this operator to the product of two such functions. 

\begin{Proposition} \label{P:leibniz-rule}
Let $f \in L$ be regular at zero and let $\sigma \in W$. Then $f_{(\sigma)}$ 
is also regular at $0$ and $f_{(\sigma)}(0) = \D{}{\sigma}{}(f)(0) = 
(\partial_{\sigma} f)(0)$. Furthermore if $g \in L$ is also regular at $0$
then
\begin{align*}
\D{}{\sigma}{0}(fg)
	&= \sum_{\rho \leq \sigma} \D{}{\rho, \sigma}{0}(f) \D{}{\rho}{0}(g)
	= \sum_{\rho \leq \sigma} \D{}{\rho}{0}(f) \D{}{\rho, \sigma}{0}(g).
\end{align*}
\end{Proposition}
\begin{proof}
Let $T \subset \Gamma$ be the set of $W$-invariant rational functions with 
nonzero constant term. This is clearly a $W$-invariant set and hence 
$T^{-1} \Gamma$ is a subalgebra of $K = \Frac(\Gamma)$. Denoting by $A$ the 
subalgebra of $L$ consisting of rational functions regular at $0$, the product 
map $T^{-1} \Gamma \ot \Lambda \to A$ is an isomorphism, since any fraction 
$p/q \in L$ with $p,q \in \Lambda$ can be rewritten so that $q \in \Gamma$. 
Thus $A$ is a free $T^{-1}\Gamma$-module with basis $\{\SS_\sigma \mid \sigma 
\in W\}$ and $f_{(\sigma)} \in A$ for all $\sigma \in W$. 

As noted in the preamble for each $\gamma \in \Gamma$ we have 
$\D{}{\sigma}{0}(\gamma f) = \gamma(0) \D{}{\sigma}{0}(f)$, and it follows 
that the same holds if $\gamma \in A^G$. Thus
\begin{align*}
\D{}{\sigma}{}(f)(0)
	&= \sum_{\tau} f_{(\tau)}(0) \D{}{\sigma}{}(\SS_\tau)(0)
	= f_{(\sigma)}(0)
\end{align*}
as stated. Analogously $\partial_\sigma$ is a $K$-linear operator, and hence
\begin{align*}
(\partial_{\sigma} f)(0)
	&= \sum_{\tau} f_{(\tau)}(0) \frac{1}{|W|}(\partial_{\sigma} 
		\partial_{\tau^{-1}\omega_0} \Delta(\Phi))(0).
\end{align*}
Now $\frac{1}{|W|}\partial_{\sigma} \partial_{\tau^{-1}\omega_0} \Delta(\Phi)$ 
is zero unless $\ell(\tau \sigma) = \ell(\tau) + \ell(\sigma)$, in which case it 
equals $\SS_{\tau\sigma^{-1}}$. The evaluation of this polynomial at $0$ is zero 
except when $\tau = \sigma$, and in this case the polynomial is just the 
constant $1 \in \CC$. Therefore, $(\partial_{\sigma} f)(0) = f_{(\sigma)}(0) 
= \D{}{\sigma}{0}(f)$.

Finally,
\begin{align*}
\D{}{\sigma}{0}(fg)
	&= \sum_{\tau, \rho} f_{(\tau)}(0) g_{(\rho)}(0) 
		\D{}{\sigma}{}(\SS_\tau \SS_\rho)(0) 
	= \sum_{\tau, \rho} c_{\tau, \rho}^\sigma \D{}{\tau}{0}(f) 
	\D{}{\rho}{0}(g),
\end{align*}
which proves the first identity in the proposition. To prove the second identity, we use that
$c^\sigma_{\tau, \rho} = c^\sigma_{\rho, \tau}$.
\end{proof}


\section{Galois orders and Gelfand-Tsetlin modules}
Throughout this section $\Gamma$ is a noetherian integral domain, $K$ is its 
field of fractions, and $L$ is a finite Galois extension of $K$ with Galois 
group $G$. Hence $K=L^G$.

\paragraph
\about{Galois orders}
We first recall the notion of a Galois ring (order), that was introduced in 
\cite{FO-galois-orders}. Let $\mathcal M$ be a monoid acting on $L$ by ring
automorphisms, such that for all $t \in \mathcal M$ and all $\sigma \in G$ we 
have $\sigma \circ t \circ \sigma^{-1} \in \mathcal M$. Then the action of $G$ 
extends naturally to an action on the smash product $L \# \mathcal M$. We 
assume that the monoid $\mathcal M$ is \emph{$K$-separating}, that is given 
$m,m' \in \mathcal M$, if $m|_K=m'|_K$ then $m = m'$. 
 
\begin{Definition}
Set $\mathcal{K} = L \# \mathcal M$.
\begin{enumerate}[(i)]
\item A \emph{Galois ring over $\Gamma$} is a finitely generated 
$\Gamma$-subring $U \subset (L \# \mathcal M)^G$ such that $UK=KU=
\mathcal{K}$. 

\item Set $S = \Gamma \setminus \{0\}$. A Galois ring $U$ over $\Gamma$ is a 
\emph{right (respectively, left) Galois order}, if for any finite-dimensional 
right (respectively left) $K$-subspace $W\subset U[S^{-1}]$ (respectively, 
$W\subset [S^{-1}]U$), the set $W\cap U$ is a finitely generated right 
(respectively, left) $\Gamma$-module. A Galois ring is \emph{Galois order} 
if it is both a right and a left Galois order.
\end{enumerate}
\end{Definition}
We will always assume that Galois rings are complex algebras. In this case 
we say that a Galois ring is a Galois algebra over $\Gamma$. 

\paragraph
\about{Principal and co-principal Galois orders}
Notice that $L \# \mathcal M$ acts on $L$, where for each $X = \sum_{m \in 
\mathcal M} l_m m \in L \# M$ we define its action on $f \in L$ by $X(f) = 
\sum_{m} l_m m(f)$.

As an example of a Galois order, Hartwig introduced \emph{the standard Galois 
$\Gamma$-order in $\mathcal K$} defined as $\mathcal K_\Gamma = \{X \in 
\mathcal K \mid X(\Gamma) \subset \Gamma\}$, see 
\cite{Hart-rational-galois}*{Theorem 2.21}. In this article the term 
``standard Galois order'' has a different meaning, and for sake of clarity 
will refer to the algebra above as \emph{the left Hartwig order of $\mathcal 
K$}. A \emph{principal Galois order} is any Galois order $U \subset \mathcal 
K_\Gamma$. By restriction $\Gamma$ is a left $U$-module for any principal 
Galois order, and hence its complex dual $\Gamma^*$ is a right $U$-module.

Denote by $\mathcal M^{-1}$ the monoid formed by the inverses of the elements 
in $\mathcal M$. Following \cite{Hart-rational-galois}, we define an 
anti-isomorphism $-^\dagger: L \# \mathcal M \to L \# \mathcal M^{-1}$ by 
$(l m)^\dagger = m^{-1} \cdot l
= m^{-1}(l) m^{-1}$ for any $l \in L, m \in \mathcal M$. The \emph{right 
Hartwig order} is thus defined as ${}_\Gamma \mathcal K = \{X \in \mathcal K
\mid X^\dagger(\Gamma) \subset \Gamma\}$, and a \emph{co-principal Galois 
order} is any Galois order contained in ${}_\Gamma \mathcal K$. Thus $\Gamma^*$
is a left $U$-module for any co-principal Galois order, with action given by 
$X \cdot \chi = \chi \circ X^\dagger$ for any $X \in U$ and $\chi \in 
\Gamma^*$.

\paragraph
\about{Gelfand-Tsetlin modules}\label{par-gt-modules}
Let $U$ be a Galois order over $\Gamma$ and let $M$ be any $U$-module. Given
$\m \in \Specm \Gamma$ we set $M[\m] = \{x \in M \mid \m^k x = 0 \mbox{ for }
k \gg 0\}$. Since ideals in $\Specm \Gamma$ are in one-to-one correspondence 
with characters $\chi: \Gamma \to \CC$ we also set $M[\chi] = \{x \in M \mid 
(\gamma - \chi(\gamma))^k x = 0 \mbox{ for all } \gamma \in \Gamma \mbox{ and 
} k \gg 0\}$. If $\chi$ is given by the natural projection $\Gamma \to \Gamma 
/ \m \cong \CC$ then $M[\m] = M[\chi]$.
\begin{Definition}
A Gelfand-Tsetlin module is a finitely generated $U$-module $M$ such that 
its restriction $M|_\Gamma$ to $\Gamma$ can be decomposed as a direct sum
$M|_\Gamma = \bigoplus_{\m \in \Specm \Gamma} M[\m]$. 
\end{Definition}
A $U$-module $M$ is a Gelfand-Tsetlin module if and only if for each $x \in M$
the cyclic $\Gamma$-module $\Gamma \cdot x$ is finite dimensional over $\CC$
\cite{DFO-GT-modules}*{\S 1.4}, which easily implies the following result. 
\begin{Lemma}
\label{L:sub-gt}
A $U$-submodule of a Gelfand-Tsetlin module is again a Gelfand-Tsetlin module.
\end{Lemma}

For every maximal ideal $\m$ of $\Gamma$ we denote by $\varphi(\m)$ the number 
of non-isomorphic simple Gelfand-Tsetlin modules $M$ for which $M[\m] \neq 0$. 
Sufficient conditions for the number $\varphi(\m)$ to be nonzero and finite 
were established in \cite{FO-fibers-gt}. 

Consider the integral closure $\overline{\Gamma}$ of $\Gamma$ in $L$. It is a 
standard fact that if $\Gamma$ is finitely generated as a complex algebra then 
any character of $\Gamma$ has finitely many extensions to characters of 
$\overline{\Gamma}$. Let $\overline{\m}$ be any lifting of $\m$ to 
$\overline{\Gamma}$, and $\mathcal M_{\m}$ be the stabilizer of $\overline{\m}$
in $\mathcal M$. Note that the monoid $\mathcal M_{\bf m}$ is defined uniquely 
up to $G$-conjugation. Thus the cardinality of $\mathcal M_{\m}$ does not 
depend on the choice of the lifting. We denote this cardinality by $|\m|$.

\begin{Theorem}{\cite{FO-fibers-gt}*{Main Theorem and Theorem 4.12}}
\label{T:theorem-extension}
Let $\Gamma$ be a commutative domain which is finitely generated as a complex 
algebra and let $U\subset (L\# \mathcal M)^{G}$ be a right Galois order 
over $\Gamma$. Let also $\m \in \Specm \Gamma$ be such that $|\m|$ is finite. 
Then the following hold.
\begin{enumerate}[(i)]
\item
\label{fiber-nontrivial} 
The number $\varphi(\m)$ is nonzero.

\item
\label{fiber-finite} 
If $U$ is a Galois order over $\Gamma$ then the number $\varphi(\m)$ is finite 
and uniformly bounded.

\item
If $U$ is a Galois order over $\Gamma$ and $\Gamma$ is a normal Noetherian 
algebra, then for every simple Gelfand-Tsetlin module $M$ the space 
$M[\m]$ is finite dimensional and bounded.
\end{enumerate}
\end{Theorem}

\section{Rational Galois orders}
Recall that $V$ is a complex vector space with an inner product. We set 
$\Lambda = S(V)$, and $L = \Frac(\Lambda)$. Recall that an element $g \in 
\mathsf{GL}(V)$ is called a \emph{pseudo-reflection} if it has finite order 
and fixes a hyperplane of codimension $1$. By definition every reflection is a 
pseudo-reflection, and the converse holds over $\RR$ but not over $\CC$, which 
is why finite groups generated by pseudo-reflections are called 
\emph{pseudo-reflection groups} or \emph{complex reflection groups}. We fix 
$G \subset \mathsf{GL}(V)$ a pseudo-reflection group. As usual the 
action of $G$ on $V$ induces actions on $\Lambda$ and $L$, and we denote by 
$\Gamma$ the algebra of $G$-invariant elements of $\Lambda$ and set $K = L^G$.

\paragraph Let $L \hookrightarrow \End_\CC(L)$ be the $\CC$-algebra morphism 
that sends any rational function $f \in L$ to the $\CC$-linear map $m_f: f' 
\in L \mapsto ff' \in L$. Although $\End_\CC(L)$ is not a $L$-algebra, it is
an $L$-vector space with $f \cdot \phi = m_f \circ \phi$ for all
$\phi \in \End_\CC(L)$. Also $G$ acts on $\End_\CC(L)$ by conjugation and 
$\sigma \cdot m_f = \sigma \circ m_f \circ \sigma^{-1} = m_{\sigma(f)}$ for 
each $\sigma \in G$, so the map $f \mapsto m_f$ is $G$-equivariant. For 
simplicity we will write $f$ for the operator $m_f$.

Given $v \in V$ we define a map $a_v: V \to V$ given by $a_v(v') = v'+v$. This
in turn induces an endomorphism of $\Lambda$, which we denote by $t_v$, given 
bt $t_v(f) = f \circ a_{v}$; we sometimes write $f(x+v)$ for $t_v(f)$. Each 
map $t_v$ can be extended to a $\CC$-linear operator on $L$ and
$t_v \circ t_{v'} = t_{v+v'}$, so $V$ acts on $L$ by automorphisms and we can 
form the smash product $L\# V$. Once again there is an algebra morphism $L \# V
\rightarrow \End_\CC(L)$, and the definitions imply that this map is 
$G$-equivariant. 

\begin{Lemma}
\label{L:translation-lemma}
Let $G$, $V$, and $L$ be as above, and let $Z \subset V$ be an arbitrary 
subset. Then the set $\{t_z \mid z \in Z\} \subset \End_\CC(L)$ is linearly
independent over $L$, and the map $L \# V \to \End_\CC(L)$ is injective.
\end{Lemma}
\begin{proof}
Put $T = \sum_{i=1}^N f_i t_{z_i}$ where $f_i \in L^\times$ and each $z_i 
\in Z$, and assume $T = 0$. Given $p \in \Lambda$ we obtain that $0 = T(p) = 
\sum_i f_i p(x+z_i)$, or, equivalently, $p(x) \sum_i f_i = \sum_i [p(x+z_i)- 
p(x)] f_i$. Let $v \in V$ be arbitrary and choose a polynomial $p$ of positive 
degree such that $p(v) = p(v + z_j)$ for all $j \neq i$ but $ p(v) + 1 = 
p(v + z_i)$. Then $0 = p(v+z_i) f_i(v)$ so $f_i(v) = 0$. Since $v$ is 
arbitrary this implies that $f_i = 0$ so the set $\{t_z \mid z \in Z\}$ is 
$L$-linearly independent. Since the morphism $L \# V \to \End_\CC(L)$ is 
$L$-linear and sends an $L$-basis of $L \# V$ to a linearly independent subset,
it must be injective.
\end{proof}

\paragraph
\about{Rational Galois orders} 
Given a character $\chi: G \to \CC^\times$, the space of relative invariants 
$\Lambda^G_\chi = \{f \in \Lambda \mid \sigma \cdot p = \chi(\sigma)p 
\mbox{ for all } \sigma \in G\} \subset \Lambda$ is a $\Lambda^G$-submodule of 
$\Lambda$. By a theorem of Stanley \cite{Hiller-coxeter-book}*{4.4 Proposition}
$\Lambda^G_\chi$ is a free $\Lambda^G$-module of rank $1$. The generator of 
$\Lambda^G_\chi$ is $d_\chi = \prod_{H \in \A(G)} (\alpha_H)^{a_H}$,
where $\A(G)$ is the set of hyperplanes that are fixed pointwise by some 
element of $G$, each $\alpha_H$ is a linear form such that $\ker \alpha_H = 
H$, and $a_H \in {\mathbb Z}_{\geq 0}$ is minimal with the property 
$\det[s_H^*]^{a_H} = \chi(s_H)$ for an arbitrary generator $s_H$ of the 
stabilizer of $H$ in $G$. Note that $a_H$ is independent on the choice of 
$s_H$, and that if $G$ is a Coxeter group then $a_H$ is either 
$1$ or $0$.

\begin{Definition}[\cite{Hart-rational-galois}*{Definition 4.3}]
A \emph{co-rational Galois order} is a subalgebra $U \subset \End_\CC(L)$
generated by $\Gamma$ and a finite set of operators $\mathcal X \subset 
(L \# V)^G$ such that for each $X \in \mathcal X$ there exists 
$\chi \in \hat G$ with $X d_\chi \in \Lambda \# V$.
\end{Definition}

Given $X \in L\# V$ we define its \emph{support} as the set of all $v \in V$ 
such that $t_v$ appears with nonzero coefficient in $X$. Note that the support 
is well-defined since the set $\{t_v \mid v \in V\}$ is free over $L$. We 
denote the support of $X$ by $\supp X$. Given a co-rational Galois order $U 
\subset (L\# V)^G$ we denote by $\Z(U)$ the additive monoid generated by 
$\{\supp X \mid X \in U\}$ in $V$. By \cite{Hart-rational-galois}*{Theorem 4.2}
$U$ is a co-principal Galois order in $(L\# \Z(U))^G$. In particular 
$\Gamma^*$ is a left $U$-module.

Let $v \in V$, let $\ev_v: \Gamma \to \CC$ be the character given by 
evaluation at $v$, and let $\m = \ker \ev_v$. Then the cyclic $U$-module $U 
\cdot \ev_v \subset \Gamma$ is a Gelfand-Tsetlin module 
\cite{Hart-rational-galois}*{Theorem 3.3}, and since $\ev_v \in U \cdot 
\ev_v[\m]$ we have a new proof that $\phi(\m) \neq 0$ for rational Galois 
orders. The following sections are devoted to study a different module 
associated to $v$, which always contains $\ev_v$ and turns out to be equal to 
$\ev_v \cdot U$ for generic $v$ (see Theorem \ref{T:simplicity}).

\section{Structure of $\Gamma$-modules associated to Postnikov-Stanley 
operators}
Throughout this section we fix a complex vector space $V$, and a root system
$\Phi$. We also fix a root subsytem $\Psi \subset \Phi$ with base
$\Omega \subset \Psi$. We denote by $G$ the Weyl group associated to 
$\Phi$ and by $W$ the one associated to $\Psi$. Like before, 
$\Lambda = S(V)$, $L = 
\Frac(\Lambda)$, 
$\Gamma = \Lambda^G$, and $K = L^G$. Since $W \subset G$, the group $W$ also 
acts on the vector spaces $\Lambda$, $\Gamma$, etc. All Schubert polynomials, 
Postnikov-Stanley operators, standard elements, etc. are defined with respect 
to the subsystem $\Psi$ and the base $\Omega$ unless otherwise stated.

\begin{Lemma}
\label{L:translation}
Let $v \in V$ and let $\pi^W: \Lambda \to \Lambda/I_W$ be the natural 
projection. Then $\pi^W(t_v(\Gamma)) = (\Lambda/I_W)^{W_v}$.
\end{Lemma}
\begin{proof}
Recall that $K$ is the fixed field of $G$ in $L$, and hence the fraction field of 
$\Gamma$. Since the extension $L^W \subset L$ is a Galois extension with 
Galois group $W$, the field $L^W t_v(K) \subset L$ must be the fixed field
of a subgroup $\widetilde W \subset W$. If $\sigma \in W_v$ and $f \in K$ then
$\sigma \cdot t_v(f) = t_{\sigma(v)} (\sigma \cdot f) = t_v(f)$, so $W_v 
\subset \widetilde W$. On the other hand, if $\sigma \in \widetilde W$, then $t_v(f)
= t_{\sigma(v)}(f)$.  So, in this case, $t_{\sigma(v) - v}(f) = f$ for all $f \in K$ and this
implies that $\sigma(v) = v$ so $\sigma \in W_v$. Thus $L^W t_v(K) = L^{W_v}$
which implies that $\Lambda^{W} t_v(\Gamma) = \Lambda^{W_v}$. Since all non-constant polynomials in $\Lambda^W$ are in the kernel of $\pi^W$ we see that
$\pi^W(\Lambda^W t_v(\Gamma)) = \pi^W(t_v(\Gamma))$, so this last space equals
$\pi^W(\Lambda^{W_v}) = (\Lambda/I_W)^{W_v}$.
\end{proof}

\paragraph
Let $\D{}{}{}(\Omega, v)$ be the complex subspace of $L^*$ (the complex dual
of $L$) spanned by $\{\D{\Omega}{\sigma}{v} \mid \sigma \in W\}$. From now on we omit the superscript $\Omega$. The generalized Leibniz rule from 
Proposition \ref{P:leibniz-rule} implies that $\D{}{}{}(\Omega, v)$ is a 
$\Lambda$-submodule of $\Hom_\CC(L,\CC)$, since for each $f \in \Lambda$ and 
$g \in L$ we have
\begin{align*}
	(f \cdot \D{}{\sigma}{v})(g)
		&= \D{}{\sigma}{0} (t_v(f)t_v(g))
		= \sum_{\tau \leq \sigma} 
			\D{}{\tau,\sigma}{0}(t_v(f)) \D{}{\tau}{0}(t_v(g)) \\
		&= \sum_{\tau \leq \sigma} \D{}{\tau,\sigma}{v}(f) \D{}{\tau}{v}(g).
\end{align*} 
Now let $\DD{}{\sigma}{v}$ denote the restriction of $\D{}{\sigma}{v}$ to 
$\Gamma$ and let $\DD{}{}{}(\Omega, v)$ be the subspace of $\Gamma^*$ spanned
by $\left\{\DD{}{\sigma}{v} \mid \sigma \in W \right\}$. We call  $\DD{}{}{}(\Omega, v)$ \emph{the space of BGG differential operators associated to $\Omega$ and $v$}. The same computation as above shows 
that $\DD{}{}{}(\Omega, v)$  is a $\Gamma$-submodule of $\Gamma^*$. We record this result in the 
following theorem.
\begin{Theorem}
\label{T:gamma-module}
Let $v \in V$. The space $\mathcal D(\Omega, v)$ is a $\Gamma$-submodule of 
$\Gamma^*$ and for each $\gamma \in \Gamma$
\begin{align*}
\gamma \cdot \DD{}{\sigma}{v}
	&= \gamma(v) \DD{}{\sigma}{v} + 
		\sum_{\tau < \sigma} \D{}{\sigma, \tau}{v}(\gamma) 
			\DD{}{\tau}{v}.
\end{align*}
\end{Theorem}

\paragraph 
\about{The structure of $\DD{}{}{}(\Omega, v)$ as $\Gamma$-module}
The modules $\DD{}{}{}(\Omega, v)$ will play an important role in our study
of Gelfand-Tsetlin modules over a co-rational Galois order. We thank David 
Speyer for pointing out the following technical result in \cite{MathOver}, 
which greatly simplified our presentation.

Recall that $W^v$ is the set of minimal length representatives of the left 
$W_v$-cosets.
\begin{Lemma}
\label{L:speyer-lemma}
Let $v \in V$ be $\Omega$-standard, let $A = (\Lambda / I_W)^{W_v}$, and let
$\omega_0^v$ be the longest element in $W^v$. Then the bilinear form $(a,b) \in A 
\times A \mapsto \D{}{\omega_0^v}{v}(ab) \in \CC$ is non-degenerate.
\end{Lemma}
\begin{proof}
By the Chevalley-Shephard-Todd theorem $\Lambda^{W_v}$ and $\Lambda^W$ are 
polynomial algebras, generated by algebraically independent sets $p_1, \ldots,
p_r$ and $q_1, \ldots, q_s$ respectively. Clearly $p_i \in \Lambda^{W_v}$ and 
$A = \CC[q_1, \ldots, q_s]/ J$, where $J$ is the ideal generated by the 
$p_i$'s. This implies that $A$ is a finite-dimensional complete intersection,
and hence a graded Artinian self-injective ring. 

Set $r = \ell(\omega_0^v)$. Then $A_n = 0$ for $n > r$, while $A_r$ is spanned
over $\CC$ by $\SS_{\omega_0^v}$ and the bilinear form in the statement is 
given by taking the coefficient of $\SS_{\omega_0^v}$ in the product $fg$.
By \cite{Lam-modules-book}*{(16.22) and (16.55)}, $A$ is a symmetric algebra 
and there exists a nonsingular associative bilinear form $B: A \times A 
\to \CC$; where by associative we mean that $B(a,bc) = B(ab,c)$ for every 
$a,b,c \in A$. To finish the proof we show that we can choose $B$ so that $B(a,b) = 
\D{}{\omega_0^v}{ab}$. Since $B$ is non-degenerate there exists $a' \in A$ 
such that $B(a',\SS_{\omega^v_0}) = 1$, and if $a'$ were of positive degree 
then $B(a',\SS_{\omega_0^v}) = B(1, a' \SS_{\omega^v_0}) = 0$ so $a' \in \CC$. 
Without loss of generality we may assume that $a' = 1$, which implies that 
$B(f,g) = B(1,fg) = \D{}{\omega_0^v}{}(fg)$.
\end{proof}
\begin{Proposition}
\label{P:gamma-module-structure}
Suppose that $v \in V$ is $\Omega$-standard. 
\begin{enumerate}[(a)]
\item 
\label{i:gamma-basis}
The set $\{\DD{}{\sigma}{v} \mid \sigma \in W^v\}$ is a basis of 
$\mathcal D(\Omega, v)$, and $\DD{}{\sigma}{v} = 0$ for all $\sigma \notin 
W^v$.

\item 
\label{i:gamma-cyclic}
Let $x = \sum_{\sigma \in W^v} a_\sigma \DD{}{\sigma}{v}$. Then $\mathcal 
D(\Omega, v) = \Gamma \cdot x$ if and only if $a_{\omega_0^v} \neq 0$.

\item
\label{i:gamma-ev-kernel}
With the same notation as as in part \ref{i:gamma-cyclic}, 
$(\gamma - \gamma(v)) x = 0$ for all $\gamma \in \Gamma$ if and only if 
$a_\sigma = 0$ for all $\sigma \neq e$.

\item 
\label{i:gamma-comparision}
Let $v' \in V$. The space $\mathcal D(\Omega, v) \cap \mathcal 
D(\Omega, v')$ is non-zero if and only if $v'$ is in the $G$-orbit of $v$. 
Furthermore if $v'$ is in the $W$-orbit of $v$ then $\mathcal D(\Omega, v) 
= \mathcal D(\Omega, v')$. 
\end{enumerate}
\end{Proposition}
\begin{proof}
Since $\DD{}{\sigma}{v}$ is a differential operator we have $\DD{}{\sigma}{v}
= \D{}{\sigma}{0} \circ t_v |_\Gamma$, so for every $\gamma \in \Gamma$
$\D{}{\sigma}{v}(\gamma) = \D{}{\sigma}{0}(t_v(\gamma))$. By the definition  of
 $\D{}{\sigma}{0}$, the latter value depends only on the image of 
$t_v(\gamma)$ modulo the ideal $I_W$, and, thanks to Lemma 
\ref{L:translation}, $\pi^W(t_v(\Gamma))$ is exactly the space of $W_v$-invariants of $\Lambda/I_W$. On the other hand, the set of Schubert polynomials $\{\SS_\sigma 
\mid \sigma \in W^v\}$ forms a basis of this space. Thus, for each $\sigma \in W^v$
there exists $\gamma_\sigma \in \Gamma$ such that $t_v(\gamma_\sigma) \equiv 
\SS_\sigma \mod I_W$ and these elements span $\pi^W(t_v(\Gamma))$. Hence
$\DD{}{\sigma}{v}(\gamma_\tau) = \delta_{\sigma,\tau}$ for all $\sigma \in W$
and $\tau \in W^v$, and this implies part (\ref{i:gamma-basis}).

For part (b), note that by Lemma \ref{L:speyer-lemma}, for each $\sigma \in W^v$ there exist
polynomials $\gamma_\sigma^*$ such that $\DD{}{\omega_v^0}{v}(\gamma_\sigma^*
\gamma_\tau) = \delta_{\sigma,\tau}$ for all $\tau \in W^v$. This implies that
$\gamma_\sigma^* \cdot \DD{}{\omega_0^v}{v} = \DD{}{\sigma}{v}$ and hence, 
if $x$ is as in the statement with $a_{\omega_0^v}
\neq 0$, then for each $\sigma \in W^v$ there exists $\gamma \in \Gamma$ such
that $\gamma \cdot x$ equals the sum of $\DD{}{\sigma}{v}$ and a linear combination of
operators $\DD{}{\tau}{v}$ with $\tau < \sigma$. This proves part 
\ref{i:gamma-cyclic}.

Let $\m = \ker \ev_{v} \subset \Gamma$. The adjointness between the $\Hom$ and 
the tensor product functors implies that $\Hom_\Gamma(\Gamma / \m, \Gamma^*) \cong 
\Hom_\CC(\Gamma/\m, \CC) \cong \CC$, so the space of elements in $\Gamma^*$
annihilated by $\m$ has complex dimension $1$. Since $\gamma \DD{}{e}{v} = 
\gamma(v) \DD{}{e}{v}$ this space is generated by $\DD{}{e}{v}$ and this 
implies (\ref{i:gamma-ev-kernel}).

It follows from the explicit formulas for the action of $\gamma \in \Gamma$ 
that each element in $\DD{}{}{}(\Omega, v)$ is a generalized eigenvector of 
$\gamma$ with eigenvalue $\gamma(v)$. Thus if $\DD{}{}{}(\Omega, v) \cap 
\DD{}{}{}(\Omega, v') \neq 0$ we must have $\gamma(v) = \gamma(v')$ for all
$\gamma \in \Gamma$ which implies that $v' \in G \cdot v$. Now if $v' = 
\tau(v)$ for some $\tau \in W$ then $\DD{}{\sigma}{\tau(v)} = \D{}{\sigma}{0}
\circ t_{\tau(v)}|\Gamma = \D{}{\sigma}{0} \circ \tau \circ t_v \circ \tau^{-1}
|_\Gamma = \D{}{\sigma}{0} \circ \tau \circ t_v$. Since $\D{}{\sigma}{0} 
\circ \tau$ lies in $\mathcal H_W$, for each $\rho
\in W$ there exist $c_\rho \in \CC$ such that $\D{}{\sigma}{0} \circ \tau = \sum_\rho c_\rho \D{}{\rho}{0}$. Hence, $\DD{}{\sigma}{\tau(v)} = \sum_\rho 
c_\rho \DD{}{\rho}{v}$, which proves part (\ref{i:gamma-comparision}).
\end{proof}

\paragraph
\about{Jordan blocks of elements in $\Gamma$}
Let $v \in V$ be $\Omega$-standard. For each $\gamma \in \Gamma$ let us denote
by $[\gamma]$ the matrix of the endomorphism of $\DD{}{}{}(\Omega, v)$ induced 
by $\gamma$ relative to the basis described in 
Proposition \ref{P:gamma-module-structure}(\ref{i:gamma-basis}) and ordered by decreasing length. By 
Theorem \ref{T:gamma-module}, $[\gamma]$ is a lower triangular matrix all diagonal entries of which equal $\gamma(v)$. Thus, the Jordan form of the matrix consists of Jordan
blocks with this eigenvalue. To provide further properties on the Jordan
form of $[\gamma]$ for generic elements of $\Gamma$cwe need the 
following lemma.

\begin{Lemma}
\label{L:differential-power}
For each $\sigma \in W$ and each $f \in \Lambda_1$ we have
$\D{}{\sigma}{0}(f^{\ell(\sigma)}) = \sum_{C(\sigma)} 
\prod_{i=1}^{\ell(\sigma)} \D{}{s_i}{0}(f)$, where the sum runs over the set $C(\sigma)$ of 
reduced expressions $\sigma = s_1s_2 \cdots s_{\ell(\sigma)}$ of $\sigma$.
\end{Lemma}
\begin{proof}
We will prove the statement by induction on $r = \ell(\sigma)$. The base 
case $r = 0$ follows from $f(0) = 0$. Now writing $f^r = f f^{r-1}$ and using
Proposition \ref{P:leibniz-rule} and the fact that $\D{}{\tau}{0}(f) = 0$
if $\ell(\tau) \neq 1$, we obtain
\begin{align*}
	\D{}{\sigma}{0}(ff^{r-1}) 
		&= \sum_{\ell(\tau) = \ell(\sigma) - 1} 
			\D{}{\tau}{0}(f^{r-1})\D{}{\tau,\sigma}{0}(f) \\
		&= \sum_{\ell(\tau) = \ell(\sigma) - 1} 
				\left(\sum_{C(\tau)}\prod_{i=1}^{\ell(\tau)} 
					\D{}{s_i}{0}(f^{r-1}) \right)
				\D{}{\tau,\sigma}{0}(f).
\end{align*} 
Now note that $\D{}{\tau,\sigma}{0} = \D{}{s}{0}$ if $\sigma = \tau s$, and otherwise $\D{}{\tau,\sigma}{0} = 0$. This completes the  proof.
\end{proof}

\begin{Theorem}\label{thm-Jordan}
Let $v \in V$ be standard and let $\gamma \in \Gamma$. Then the Jordan form
of the matrix $[\gamma]$ consists of Jordan blocks of size at most 
$\ell(\omega_0^v)+1$ and eigenvalue $\gamma(v)$. Furthermore, there is at most
one block of this maximal size, and for a generic element $\gamma$ of $\Gamma$ there 
is exactly one such block. 
\end{Theorem}
\begin{proof}
Set $r = \ell(\omega_0^v)$.
The formula for the action of $\Gamma$ given in Theorem \ref{T:gamma-module}
implies that $(\gamma - \gamma(v)) \DD{}{\sigma}{v}$ is a linear combination of
$\D{}{\tau}{v}$ with $\ell(\tau) < \ell(\sigma)$. It follows that 
\begin{align*}
(\gamma - \gamma(v))^{\ell(\sigma) + 1} \DD{}{\sigma}{v} &= 0
\end{align*}
so $\gamma(v)$ is the only possible eigenvalue of $\gamma$ acting on the 
space $\DD{}{}{}(\Omega, v)$, and $\ker (\gamma - \gamma(v))^{r}$ is contained 
in the linear span of $\DD{}{\omega_0^v}{v}$. This proves that all Jordan
blocks are of size at most $r+1$, and that there is at most one
block of this size. We next show that the Jordan form of 
$[\gamma]$ has generically one such block.

Denote by $N$ the subset of $\Gamma /\ann \DD{}{}{}(\Omega, v)$ consisting of
the coclasses of those $\gamma \in \Gamma$ whose Jordan form contains only 
blocks of size strictly smaller than $r + 1$. Equivalently, 
this is the set of coclasses of $\gamma$ such that $(\gamma - 
\gamma(v))^{r} \DD{}{}{}(\Omega, v) = 0$, and this set is a 
Zariski closed subset of $\Gamma / \ann \DD{}{}{}(\Omega, v)$. Now let $S_v \subset 
W_v$ be the set of all simple transpositions in $W_v$ and let $\SS = 
\sum_{s \in S_v} \SS_s \in (\Lambda/I_W)^{W_v}$. Furthermore, let $\gamma \in \Gamma$
be such that $\pi \circ t_v(\gamma) = \SS$, which exists by Lemma 
\ref{L:translation}. Then $\gamma(v) = 0$ and
\begin{align*}
\gamma^{r} \cdot \DD{}{\omega_0^v}{}(v)
	&= \DD{}{e,\omega_0^v}{v}(\gamma^{r}) \DD{}{e}{v}
	= \D{}{\omega_0^v}{0}(\SS^{r}) \DD{}{e}{v}.
\end{align*}
Now, by Lemma \ref{L:differential-power}, we have 
\begin{align*}
\D{}{\omega_0^v}{0}(\SS^{r})
	&= \sum_{C} \prod_{i=1}^{r} \D{}{s_i}{0}(\SS)
	= \sum_{C} \prod_{i=1}^{r} \mathcal I_{S_v}(s_i)
\end{align*}
where the sum is over all reduced decompositions $s_1 \cdots s_{r}$ of 
$\omega_0^v$ and $\mathcal I_{S_v}$ is the indicator function of the set $S_v$ 
(that is, $\mathcal I_{S_v}$ is $1$ over $S_v$ and $0$ over the complement of 
$S_v$). Thus, the product $\prod_{i=1}^{r} \mathcal I_{S_v}(s_i)$ is zero 
unless each $s_i$ in the reduced decomposition lies in $S_v$. In view of 
\cite{Hump-coxeter-book}*{1.10 Proposition (b)}, there is at least one
such reduced decomposition and hence $\D{}{\omega_0^v}{0}(\SS^{r}) \in 
\ZZ_{>0}$. This shows that $\gamma \notin N$ and hence $N$ is a Zariski closed 
proper subset of $\Gamma/ \ann \DD{}{}{}(\Omega, v)$. Thus the complement of 
$N$ is dense.
\end{proof}

\section{Action of a co-rational Galois order}
In this section $G$ is a reflection group acting on $V$, and hence on 
$\Lambda = S(V)$ and on its field of rational functions $L = 
\Frac(\Lambda)$. We fix a co-rational Galois order $U \subset (L\# V)^G$ and 
denote by $\Z \subset V$ the additive monoid generated by $\supp U$. 

We assume again that $\Phi \subset V^*$ is a root system with base $\Sigma$ 
and $G = W(\Phi)$. We denote by $\Psi$ a standard subsystem with base $\Omega
\subset \Sigma$ and set $W = W(\Psi)$. All Schubert polynomials and 
Postnikov-Stanley differential operators appearing in this section are defined 
with respect to $\Omega$ unless otherwise stated.

\paragraph
Recall that for each $\sigma \in G$ we introduced a divided difference 
operator as an element of the smash product $L \# G$. Since $\End_\CC(L)$ is 
an $(L\# G)$-module, given $X \in \End_\CC(L)$ and $\sigma \in G$, we obtain a 
new operator on $L$ by taking $\partial_\sigma(X)$. Notice that, in general, 
this operator is \emph{different} from the composition of $\partial_\sigma$ 
(regarded as an element of $\End_\CC(L)$) and $X$. In the following lemma we 
collect some properties of these operators.
\label{L:dd-varia}
\begin{Lemma}
Let $X \in \End_\CC(L)$.
\begin{enumerate}[(a)]
\item 
\label{i:dd-on-operators}
For each $\sigma \in G$ we have $\partial_\sigma(X)|_K = \partial_\sigma 
\circ X |_K$.

\item
\label{i:diff-dd-composition}
Let $v \in V$ be $\Omega$-standard. If $\sigma \in W^v$ and $\tau \in W_v$ 
then 
\begin{align*}
\D{}{\sigma}{v} \circ 
\partial_\tau = 
\begin{cases} 
	\D{}{\sigma\tau}{v} 
		& \mbox{ if } \ell(\sigma\tau) = \ell(\sigma) +\ell(\tau);\\ 
	0 & \mbox{otherwise.}\end{cases}
\end{align*}

\item 
\label{i:symmetrizing}
Let $\widetilde \Psi \subset \Psi$ be a standard subsystem, $W_\theta \subset 
W$ be the corresponding parabolic subgroup, $\omega_0^\theta$ be the 
longest word in $W^\theta$, and $\Delta(\Psi)^\theta := \Delta(\Psi) / 
\Delta(\widetilde \Psi)$. If $X \in \End_\CC(L)^{W_\theta}$, then 
\begin{align*}
\sum_{\sigma \in W} \sigma \cdot X = |W_\theta| \partial_{\omega_0^\theta} 
(X \Delta(\Psi)^\theta).
\end{align*}
\end{enumerate}
\end{Lemma}
\begin{proof}
We prove part (\ref{i:dd-on-operators}) by induction on $\ell(\sigma)$. If 
$\sigma$ is the identity then the result is obvious. Assume now that $\sigma 
= s \tau$ with $\ell(\sigma) = 1 + \ell(\tau)$ and $s \in S$, and that the 
statement holds for $\tau$. Setting $X' = 
\partial_\tau(X)$, we obtain
\begin{align*}
\partial_\sigma(X) (f)
	&= \partial_s(X')(f)
	= \frac{1}{\alpha_s} (X'(f) - s\circ X' \circ s (f)) 
	= \frac{1}{\alpha_s} (X'(f) - s(X'(f)))\\
	&= \partial_s(X'(f))
	= \partial_s(\partial_\tau(X(f)))
	= \partial_\sigma(X(f))
\end{align*} 
which is the desired indentity.

We now prove part (\ref{i:diff-dd-composition}). The fact that $\tau \in W_v$
implies that $t_v \circ \partial_\tau = \partial_\tau \circ t_v$. Now recall
from Proposition \ref{P:leibniz-rule} that $\D{}{\sigma}{0} = \ev_0 \circ 
\partial_\sigma$, so 
\begin{align*}
\D{}{\sigma}{v} \circ \partial_\tau 
	&= \D{}{\sigma}{0} \circ \partial_\tau \circ t_v \\
	&= \ev_0 \circ \partial_\sigma \circ \partial_\tau \circ t_v
	= \begin{cases}
		\ev_0 \circ \partial_{\sigma\tau} \circ t_v = \D{}{\sigma\tau}{v}
			& \mbox{ if } \ell(\sigma \tau) = \ell(\sigma) + \ell(\tau); \\
		0 & \mbox{otherwise}.
	\end{cases}
\end{align*}

Finally we prove part (\ref{i:symmetrizing}).
The statement of \cite{Hiller-coxeter-book}*{Chapter IV (1.6)} implies that
$\partial_{\omega_0} = \frac{1}{\Delta(\Phi)} \sum_{\sigma \in G} 
(-1)^{\ell(\sigma)} \sigma$ as operators on $L$, and since the map $L \# 
G \to \End_\CC(L)$ is injective, the identity holds in $L \# G$. Using that and the 
fact that $\sigma \cdot \Delta(\Phi) = (-1)^{\ell(\sigma)} \Delta(\Phi)$ we 
deduce that $\sum_{\sigma \in G} \sigma \cdot X = \partial_{\omega_0}
(X \Delta(\Phi))$ for any $X \in \End_\CC(L)$. Certainly, the analogous identity 
holds if we replace $G$ by any subgroup and $\Phi$ by the corresponding root 
subsystem.

Let $\omega_0$ and $\omega_1$ be the longest elements of $W$ and $W_\theta$, 
respectively. Then $\omega_0 \omega_1^{-1} \in \omega_0 W_\theta$ and its 
length equals to $\ell(\omega_0) - \ell(\omega_1)$,  the 
smallest possible length of an element in the coset $\omega_0 G_\theta$. Thus 
$\omega_0^\theta = \omega_0 \omega_1^{-1}$ and
\begin{align*}
\sum_{\sigma \in W} \sigma \cdot X 
	&=\partial_{\omega_0}(X \Delta(\Psi))
	= \partial_{\omega^\theta_0} \partial_{\omega_1}(X \Delta(\widetilde \Psi) 
	\Delta(\Psi)^\theta).
\end{align*}
Now both $\Delta(\Psi)^\theta$ and $X$ are $W_\theta$-invariant, so the last 
expression equals
\begin{align*}
\partial_{\omega^\theta_0}(
	X \Delta(\Psi)^\theta \partial_{\omega_1}(
		\Delta(\widetilde \Psi)))
	= |W_\theta| \partial_{\omega^\theta_0}(X \Delta(\Psi)^\theta),
\end{align*}
which completes the proof.
\end{proof}

Recall that for each $z \in V$ there exists some $\Omega$-standard element in 
the orbit $W \cdot z$. Thus, given $Z \subset V$ that is stable by the action of $W$, we 
can choose a set of $\Omega$-standard representatives of $Z/W$. The following
proposition shows how this fact can be used to express
elements of $U$ in different ways.

\begin{Proposition}
\label{P:form}
Let $X \in (L \#V)^G$ and assume that there exists $\chi \in \hat G$ such that 
$d_\chi X \in \Lambda \# V$. 
\begin{enumerate}[(a)]
\item 
\label{i:invariant-form}
For each $z \in \supp X$ there exists $f_z \in \Lambda^{G_z}$ such that
\begin{align*}
X
	&= \sum_{z \in \supp X} \frac{f_z}{d_{\chi}^z} t_z,
\end{align*}
where $d_{\chi}^z$ is the product of all $\alpha \in \Phi^+$ dividing $d_\chi$ 
such that $\alpha(z) \neq 0$. 

\item
\label{i:dd-form}
Let $Y$ be a set of $\Omega$-standard representatives of $\supp X / W$, and 
for each $y \in Y$ denote by $\omega_0^y$ the longest element in $W^y$, and by 
$\Delta(\Psi)^y$ the product of all roots in $\Psi^+$ with $\alpha(y) 
\neq 0$. Then 
\begin{align*}
X
	&= \sum_{y \in Y} \frac{1}{|W_y|}
		\partial_{\omega_0^{y}}\left(
			\frac{f_y \Delta(\Psi)^y}{d_{\chi}^y} t_y
		\right).
\end{align*}
\end{enumerate}
\begin{proof}
Fix $z \in \supp X$ and let $h$ be the coefficient of $t_z$ in $X$, which is 
well defined by Lemma \ref{L:translation-lemma}. Since $X$ is $G$-invariant we 
know that $\sigma \cdot X = X$ for any $\sigma \in G_z$, so $\sigma(h) = h$. 
Writing $h = \frac{g}{d_\chi}$ we have
\begin{align*}
\frac{g}{d_\chi} 
	= \sigma \cdot \frac{g}{d_\chi} 
	=\frac{\sigma \cdot g}{\chi(g) d_\chi}.
\end{align*}
Therefore, $\sigma \cdot g = \chi(\sigma) g$ for all $\sigma \in G_z$. 

Denote by $\chi'$ the restriction of $\chi$ to $G_z$. Observe that $G_z$ is the 
reflection group generated by the reflections fixing $z$ and it acts on 
$\Lambda$ by restriction. Thus,  by Stanley's theorem, the space of relative 
invariants $\Lambda^{G_z}_{\chi'}$ is generated over $\Lambda^{G_z}$ by 
$d_{\chi'}$, and this polynomial is the product of all roots $\alpha \in 
\Phi^+$ dividing $d_\chi$ such that $\alpha(z) = 0$. Therefore, $g = f_z d_{\chi}$
for some $f_z \in \Lambda^{G_z}$, which implies that $\frac{g}{d_{\chi}} = 
\frac{f_z}{d_\chi/d_{\chi'}} = \frac{f_z}{d_{\chi}^z}$. This proves part 
(\ref{i:invariant-form}). 

Since $X$ is $G$-invariant, it is clear that
\begin{align*}
X 
	&= \frac{1}{|W|} \sum_{\sigma \in W} \sigma \cdot X
	=\sum_{y \in Y} \frac{1}{|W|}\sum_{\sigma \in W} \sigma \cdot 
		\left( \frac{f_y}{d_\chi^y} t_y\right).
\end{align*}
As we mentioned before, the coefficient of $t_y$ is $G_y$-invariant, and hence 
it is $W_y$-invariant. After applying Lemma 
\ref{L:dd-varia}(\ref{i:symmetrizing}) to $W$, we obtain
\begin{align*}
\sum_{\sigma \in W} \sigma \cdot
		\left(\frac{f_y}{d_\chi^y} t_y \right)
		&= |W^y| \partial_{\omega_0^{y}}\left(
			\frac{f_y \Delta(\Psi)^y}{d_\chi^y} t_y 
		\right)
\end{align*}
and the result follows.
\end{proof}
\end{Proposition}

\paragraph
\about{$U$-submodule of $\Gamma^*$ associated to $v$}
Recall that to each $v \in V$ we associate the character $\ev_v: \Gamma \to 
\CC$ given by evaluation at $v$. Since $\Gamma$ consists of $G$-symmetric 
polynomials, $\ev_v = \ev_{\sigma(v)}$ for any $\sigma \in G$, so we can 
assume that $v$ is $\Omega$-standard. Furthermore, note that $\ev_v = 
\DD{}{e}{v} \in \DD{}{}{}(\Omega, v) \subset \Gamma^*$.

\begin{Definition}
Let $v \in V$ be standard. We denote by $V(\Omega, T(v))$ the space 
$\displaystyle \sum_{z \in \Z} \DD{}{}{}(\Omega, v+z)$.
\end{Definition}

Recall that $\Phi_0(v)$ is the set of all roots in $\Phi$ such that $\alpha(v)
= 0$. The following theorem shows that under certain conditions the space 
$V(\Omega, T(v))$ is a $U$-module. This theorem generalizes 
\cite{EMV-orthogonal}*{Theorem 10} and 
\cite{RZ-singular-characters}*{5.6 Theorem} to rational Galois orders.
\begin{Theorem}
\label{T:module-structure}
Let $v \in V$ be standard and assume that $\Phi_0(v+z) \subset \Psi$ for each 
$z \in \Z$. Then $V(\Omega, T(v)) \subset \Gamma^*$ is a Gelfand-Tsetlin 
$U$-module.
\end{Theorem}
\begin{proof}
By Theorem \ref{T:gamma-module}, the action of $\Gamma$ on $V(\Omega, T(v))$ is
locally finite, so we only need to show that it is a $U$-submodule of 
$\Gamma^*$.
By definition, $U$ is generated by a finite set $\mathcal X$ such that 
any element $X \in \mathcal X^\dagger$ satisfies the hypothesis of Proposition 
\ref{P:form}. Hence it is enough to prove the following: for each $z' \in \Z$, 
each $\sigma \in G$, and each $X$ satisfying the hypothesis of Proposition 
\ref{P:form}, we have $\DD{}{\sigma}{v+z'} \circ X \in V(\Omega, T(v))$. We 
will prove this in several steps. 

First, let $v'$ be a standard element in the $W$-orbit of $v + z'$. Since 
$\DD{}{}{}(\Omega, v') = \DD{}{}{}(\Omega, v+ z')$ by Proposition 
\ref{P:gamma-module-structure}(\ref{i:gamma-comparision}),
the statement in the theorem is equivalent to showing that $\DD{}{\sigma}{v'} 
\circ X \in V(\Omega, T(v))$. Now let $\widetilde W = W_{v'}$ and let 
$\widetilde \Psi = \Psi_0(v')$ be 
the associated standard root subsystem. By Proposition 
\ref{P:form}(\ref{i:dd-form}), $X$ can be written as a sum of operators of the form 
$\partial_{\widetilde \omega_0^z}(F_z t_z)$ for $z \in \Z$, where $\widetilde 
\omega_0^z$ is the longest element of $\widetilde W^z$ and $F_z = \frac{f_z 
\Delta(\widetilde \Psi)^z}{d_\chi^z}$. Thus 
\begin{align*}
\DD{}{\sigma}{v'} \circ X
	&= \sum_{z \in Y} \frac{1}{|\widetilde W_z|}
		\D{}{\sigma}{v'} \circ \partial_{\widetilde \omega_0^{z}}(
			F_z t_z)|_\Gamma,
\end{align*}
where $Y$ is a set of $\widetilde \Omega$-standard representatives of $\supp X / 
\widetilde W$. So, it is enough to show that $\D{}{\sigma}{v'} \circ 
\partial_{\widetilde \omega_0^z}(F_z t_z)|_\Gamma \in V(\Omega, T(v))$ for any $z \in 
Y$.

We claim that $F_z$ is regular at $v'$. Recall that $d_\chi^z$ is the product
of all roots $\alpha_s$ such that $\chi(s) = -1$ and $\alpha_s(z) \neq 0$. If
one of this factors is such that $\alpha_s(v') = 0$ then $\alpha_s \in 
\Phi_0(v') = \Phi_0(\tau(v+z')) = \tau(\Phi_0(v+z'))$ for some $\tau \in W$. 
Now since $\Phi_0(v+z') \subset \Psi$ by hypothesis, and since $\Psi$ is stable
by the action of $W$, it follows that $\Phi_0(v') \subset \Psi$, and hence 
$\alpha_s$ is also a factor of $\Delta(\widetilde \Psi)^z$. Thus the term
$\Delta(\widetilde \Psi)^z$ in the numerator cancels out all the linear terms in
the denominator which are zero at $v'$. This proves that $F_z$ is regular at $v'$.

We make one further simplification. By parts (\ref{i:dd-on-operators}) and
(\ref{i:diff-dd-composition}) of Lemma \ref{L:dd-varia},
\begin{align*}
\D{}{\sigma}{v'} \circ 
	\partial_{\widetilde \omega_0^z}(F_z t_z) |_\Gamma
	&= \D{}{\sigma}{0} \circ t_{v'} \circ \partial_{\widetilde \omega_0^z} 
		\circ F_z t_z |_\Gamma
	= \D{}{\sigma}{0} \circ \partial_{\widetilde \omega_0^z} \circ t_{v'}(F_z) 
		t_{v'+z} |_\Gamma \\
	&= \begin{cases}
		\D{}{\sigma \widetilde \omega_0^z}{0} \circ t_{v'}(F_z) t_{v'+z}
			& \mbox{ if } \ell(\sigma \widetilde \omega_0^z) = \ell(\sigma)
				+ \ell(\widetilde \omega_0^z);\\
		0 & \mbox{otherwise.}
	\end{cases}
\end{align*}
Here we have used that $t_v$ and $\partial_{\widetilde \omega_0^z}$ commute since
$\widetilde \omega_0^z \in W_{v'}$. If the result above is $0$ then we are done. On the 
other hand, since $F_z$ is regular at $v'$ then $t_{v'}(F_z)$ is regular at 
$0$. So, writing $t_{v'}(F_z) = \sum_{\rho \in W} (t_{v'}(F_z))_{(\rho)} 
\SS_\rho$ and recalling from Proposition \ref{P:leibniz-rule} that 
$(t_{v'}(F_z))_{(\rho)}(0) = \D{}{\rho}{0}(t_{v'}(F_z)) = \D{}{\rho}{v'}(F_z)$,
we obtain that
\begin{align*}
\D{}{\sigma}{v'} \circ \partial_{\widetilde \omega_0^z} (F_z t_z)|_\Gamma
	&= \begin{cases}
		\sum_{\rho \in W} 
			\D{}{\rho}{v'}(F_z)
				(\D{}{\sigma \widetilde \omega_0^z}{0} \circ 
					\SS_\rho t_{v'+z})|_\Gamma
			& \mbox{ if } \ell(\sigma \widetilde \omega_0^z) = \ell(\sigma)
				+ \ell(\widetilde \omega_0^z);\\
		0 & \mbox{otherwise.}
	\end{cases}
\end{align*}
Finally, let $\gamma \in \Gamma$. Then 
\begin{align*}
(\D{}{\sigma \widetilde \omega_0^z}{0} \circ \SS_\rho t_{v'+z}) (\gamma)
	&= \D{}{\sigma \widetilde \omega_0^z}{0}(\SS_\rho t_{v'+z}(\gamma)) \\
	&= \sum_{\nu \in W} t_{v'+z}(\gamma)_{(\nu)}(0)
		\D{}{\sigma \widetilde \omega_0^z}{0} (\SS_{\rho} \SS_{\nu}) \\
	&= \sum_{\nu \in W}
		c^{\sigma \widetilde \omega_0^z}_{\rho, \nu} \D{}{\nu}{v+z'}(\gamma).
\end{align*} 
Using the identities above,  we obtain
\begin{align*}
\DD{}{\sigma}{v'} \circ X 
	&= \sum_{\substack{z \in Y \\ \ell(\sigma \widetilde \omega_0^z)
		= \ell(\sigma) + \ell(\widetilde \omega_0^z)}} \frac{1}{|\widetilde W_z|}
		\sum_{\rho, \nu \in W} 
			c^{\sigma \widetilde \omega_0^z}_{\rho, \nu}
				\D{}{\rho}{v'}(F_z) \DD{}{\nu}{v'+z} \\
	&= \sum_{\substack{z \in Y \\ \ell(\sigma \widetilde \omega_0^z)
		= \ell(\sigma) + \ell(\widetilde \omega_0^z)}} \frac{1}{|\widetilde W_z|}
		\sum_{\nu \in W} 
			\D{}{\nu, \sigma \widetilde \omega_0^z}{v'}(F_z) \DD{}{\nu}{v'+z}.
\end{align*}
Now $v' + z = \tau(v+z') + z = \tau(v + z' + \tau^{-1}(z))$ and hence 
$\DD{}{\nu}{v'+z} \in V(\Omega, T(v))$.
\end{proof}

\section{Standard Galois orders of type $A$}
In this section we consider a special type of Galois order, for which we find 
a basis of Postnikov-Stanley operators for the module introduced in Theorem 
\ref{T:module-structure}. We also give a sufficient condition for the 
simplicity of this module.

\paragraph
Given $\mu = (\mu_1, \ldots, \mu_r) \in \NN^r$ we set $\CC^\mu = \CC^{\mu_1} 
\times \cdots \times \CC^{\mu_r}$ and $\II = \II(\mu) = \{(k,i) \mid 1 \leq k 
\leq r, 1 \leq i \leq \mu_k\}$. Also, for each $v \in \CC^\mu$ and $(k,i) \in 
\II$, we will denote by $v_k$ the projection of $v$ to the component 
$\CC^{\mu_k}$, and by $v_{k,i}$ the $i$-th coordinate of $v_k$. We will denote 
by $e_{k,i}$ the vector of $\CC^\mu$ with $(e_{k,i})_{l,j} = \delta_{k,l}
\delta_{i,j}$, and refer to the set $\{e_{k,i} \mid (k,i) \in \II\}$ as the 
\emph{canonical basis} of $\CC^\mu$. We denote by $\{x_{k,i} \mid (k,i) \in 
\II\}$ the dual basis to the canonical basis, so $\CC[X_\mu] = \CC[x_{k,i} 
\mid (k,i) \in \II]$ is the algebra of polynomial functions over $\CC^\mu$.
We denote the fraction field of this algebra by $\CC(X_\mu)$. For each $(k,i) 
\in \II$ we write $t_{k,i}$ for the automorphism $t_{e_{k,i}} \in 
\End_\CC(\CC(X_\mu))$.

For each $1 \leq j \leq r$ the symmetric group $S_{\mu_j}$ acts on 
$\CC^{\mu_j}$ by permuting the coordinates of a vector, and hence $S_\mu = 
S_{\mu_1} \times \cdots \times S_{\mu_r}$ acts on $\CC^\mu$. This is a 
reflection group corresponding to the root system $\Phi = \{x_{k,i} - x_{k,j} 
\mid (k,i),(k,j) \in \II\}$. We fix $\Sigma = \{x_{k,i} - x_{k,i+1} \mid 1 
\leq k \leq r, 1 \leq i < \mu_k\}$  as a base of $\Phi$.
Given $\sigma \in S_{\mu}$ we will denote by $\sigma[k]$ its projection to
$S_{\mu_k}$. Also, given $\tau \in S_{\mu_k}$ we will denote by $\tau^{(k)}$
the unique element of $S_\mu$ such that $\tau^{(k)}[k] = \tau$ and 
$\tau^{(k)}[l] = \Id_{S_{\mu_l}}$ for $l \neq k$. We denote by 
$\sym_k = \frac{1}{\mu_k!}\sum_{\sigma \in S_{\mu_k}} \sigma^{(k)} \in 
\CC[S_\mu]$, and $\Delta_k= \prod_{1 \leq i < j \leq \mu_k}(x_{k,i}-x_{k,j})$. 
Notice that $\Delta_k$ is the generator of the space of relative invariants
associated to the character $\sg[k]$ given by $\sg[k](\sigma)=\sg(\sigma[k])$.

The action of $S_\mu$ on $\CC^\mu$ induces actions on $\CC[X_\mu]$ and 
$\CC(X_\mu)$, so we may consider Galois orders in $(\CC(X_\mu) \# 
\CC^\mu)^{S_\mu}$. The following definition distinguishes a special class of 
such rational Galois orders.

\begin{Definition}
Let $\mu = (\mu_1, \ldots, \mu_r) \in \NN$ and let $U \subset (\CC(X_\mu) \# 
\CC^\mu)^{S_\mu}$ be a Galois order. We will say that $U$ is a \emph{standard 
Galois order of type $A$} if it is generated by $\CC[X_\mu]^{S_\mu}$ and a set 
$\mathcal X = \{X_k^\pm \mid 1 \leq k \leq r'\}$ for some $r' \leq r$ such that
\begin{align*}
X_k^\pm 
 &= \sym_k \left( 
  t_{\pm e_{k,1}}
 \frac{f_k^\pm}{\prod_{j = 2}^{\mu_k}(x_{k,1} - x_{k,j})} 
 \right).
\end{align*}
\end{Definition}
\begin{Remark*}
As indicated earlier, by definition, a standard Galois order of type $A$ is 
not necessarily a standard Galois order in the sense of Hartwig's definition, 
in \cite{Hart-rational-galois}*{Definition 2.30}. 
\end{Remark*}

Notice that in the definition above $X_k^\pm \Delta_k  \in \CC[X_\mu] \# 
\ZZ^\mu$ so $U$ is a co-rational Galois order. From now on set $\overline \mu = 
(\mu_1, \ldots, \mu_{r'}, 0, \ldots, 0) \in \NN^r$.  By definition,   
 $\supp U = \ZZ^{\overline \mu}$ for any $U$ which is a standard 
Galois order of type $A$. 

\begin{Example*}
As discussed in \cite{Hart-rational-galois}*{\S 4.2}, finite $W$-algebras of
type $A$ are co-rational Galois orders. The explicit formulas given in that
paragraph show that they are in fact standard Galois orders of type $A$.
Simmilarly the formulas from \cite{Hart-rational-galois}*{\S 4.4} show that
orthogonal Gelfand-Tsetlin algebras, introduced by Mazorchuk in 
\cite{Maz-orthogonal-GT-alg}, are also examples of standard Galois orders
of type $A$.
\end{Example*}

\paragraph
\about{Modules of the form $V(\Omega, T(v))$}\label{par-mod-omega}
Fix $\mu \in \NN^r$ and let $U \subset (\CC[X_\mu] \# \CC^\mu)^{S_\mu}$ be a
standard Galois order of type $A$. We will denote by $\overline \Phi$ the root 
system $\{x_{k,i} - x_{k,j} \mid 1 \leq k \leq r', 1 \leq i < j \leq \mu_k\}$, 
and by $\overline \Sigma$ the base $\Sigma \cap \overline \Phi$. 

Given $v \in \CC^\mu$ we set $\Psi(v) = \{\alpha \in \overline \Phi \mid 
\alpha(v) \in \ZZ\}$. We will say that $v$ is a \emph{seed} if $\Psi$ is a
standard root subsystem of $\overline \Phi$ and $\Psi(v) = \overline 
\Phi_0(v)$; notice that this second equality is equivalent to $W_v = 
W(\Psi(v))$. We claim that for every element $v \in \CC^\mu$ there exists a 
seed $\vv$ of the form $\sigma(v) + z$ for some $z \in \ZZ^{\overline \mu}$ 
and some $\sigma \in S_{\overline \mu}$. Indeed, since $\Psi(v)$ is a root 
subsystem of $\overline \Phi$, there exists $\sigma \in S_{\overline \mu}$ 
such that $\sigma(\Psi(v)) = \Psi(\sigma^{-1}(v))$ is a standard subsystem. In 
other words, $v' = \sigma^{-1}(v)$ has the property that if $v'_{k,i} - 
v'_{k,j} \in \ZZ$ for some $(k,i), (k,j) \in \II(\overline \mu)$ with $i < j$, 
then $v'_{k,s} - v'_{k,s+1} \in \ZZ$ for any $i \leq s < j$. It follows that 
there exists $z \in \ZZ^{\overline \mu}$ such that $v'' = v'+z$ has an even 
stronger property: if $v''_{k,i} - v''_{k,j} \in \ZZ$ for some $(k,i), (k,j) 
\in \II(\overline \mu)$ with $i < j$, then $v''_{k,s} = v''_{k,s+1}$ for any 
$i \leq s < j$, or equivalently $v''$ is seed.  

Fix a seed $\vv$, and set $\Psi = \Psi(\vv)$ and $\Omega = \Psi(\vv) \cap
\overline \Sigma$. We denote by $\Z(\vv)$ the set of all $z \in \ZZ^{\overline 
\mu}$ such that $\alpha(z) \geq 0$ for all $\alpha \in \Omega$. This are the 
integral points in the fundamental domain of the system $\Omega$ seen as a root
system over the real vector space $\RR^\mu \subset \CC^\mu$, see
\cite{Hump-coxeter-book}*{\S 1.12}. Also,  for each $z \in \Z(\vv)$, we define
an equivalence relation $\sim_z$ on $\II(\overline 
\mu)$, by letting $(k,i) \sim_z (l,j)$ if and only if $l = k$ and 
$(\vv + z)_{k,i} = (\vv + z)_{k,j}$. Denote by $\II(\overline \mu, z)$
the set of all equivalence classes of this equivalence relation. Each 
equivalence class $I \in \II(\overline
\mu, z)$ is by definition a set of the form $\{(k,i), (k,i+1), \ldots, (k,j)\}$
for some $1 \leq i < j \leq \mu_k$. We will write $a^+(I)$ for $(k,i)$ and 
$a^-(I)$ for $(k,j)$, i.e.  the first and last elements of $I$, respectively,  
with respect to the lexicographic order. 
\begin{Lemma}
\label{L:seed-varia}
Let $\vv \in \CC^\mu$ be a seed, $\Psi = \Psi(\vv), \Omega = \Psi \cap 
\overline \Sigma$ and $W = W(\Psi)$.
\begin{enumerate}[(i)]
\item
\label{i:standard}
If $z \in \ZZ^\mu$ then $\vv + z$ is $\Omega$-standard if and only if $z \in 
\Z(\vv)$.

\item
\label{i:unique}
If $z, z' \in \Z(\vv)$ and $\vv + z = \sigma(\vv + z')$ for some $\sigma \in
S_{\mu}$, then $z = z'$.

\item
\label{i:plus-minus}
If $z \in \Z(\vv)$, then $z \pm \delta^{k,i} \in \Z(\vv)$ if and only if
$(k,i) = a^\pm(I)$ for some $I \in \II(\overline \mu, z)$.
\end{enumerate}
\end{Lemma}
\begin{proof}
The definition of a seed implies that $\alpha \in \Psi$ if and only if 
$\alpha(\vv) = 0$. Now $\vv + z$ is $\Omega$-standard if and only if 
$\alpha(\vv + z) = \alpha(z) \geq 0$ for all $\alpha \in \Omega$. Hence part
(\ref{i:standard}) follows immediately from these definitions. 

Let now $\vv + z = \sigma(\vv + z')$; since $z \in \ZZ^{\overline \mu}$, we 
can assume that $\sigma \in S_{\overline \mu}$. Then $\vv - \sigma(\vv) = 
\sigma(z') - z$, and so $\vv_{k,i} - \sigma(\vv)_{k,i} = \vv_{k,i} - \vv_{
k, \sigma[k]^{-1}(i)} \in \ZZ$ for all $(k,i) \in \II(\overline \mu)$. By the 
definition of a seed this is possible if and only if $\sigma(\vv) = \vv$, so 
$z = \sigma(z')$. As mentioned above, $z, z' \in \Z(\vv)$ is equivalent to the
property that $\alpha(z), \alpha(z') \in \ZZ_{\geq 0}$ for all $\alpha \in 
\Omega$, and by \cite{Hump-coxeter-book}*{1.12 Theorem, part (a)} there is 
exactly one element in $W\cdot z$ with this property, so $z = z'$ and part 
(\ref{i:unique}) is proved. 

Finally it is easy to check that $z \in \Z(\vv)$ if and only if for each $I' 
= \{(k,i'), (k,i'+1), \ldots, (k,j')\} \in \II(\overline \mu, \vv)$ we have 
$z_{k,i'} \geq z_{k,i'+1} \geq \cdots \geq z_{k,j'}$. Thus if $z + 
\delta^{k,i} \in \Z(\vv)$ then either $i = 1$ or $z_{k,i-1} > z_{k,i} \geq 
z_{k,i+1}$. In either of the two cases there exists $I \in \II(\overline \mu, 
\vv + z)$ with $a^+(I) = (k,i)$. A similar argument shows that if $z - 
\delta^{k,i} \in \Z(\vv)$ then there must exist an $I$ such that $a^-(I) = 
(k,i)$ and part (\ref{i:plus-minus}) is proved.
\end{proof}

We are now ready to prove the following result that generalizes 
\cite{RZ-singular-characters}*{5.6 Theorem} and 
\cite{EMV-orthogonal}*{Theorem 10} to integral Galois algebras of type $A$. 
For the sake of comparison, we note that the sets $\Z(\vv)$ and $W^z$ in the
following theorem correspond respectively to the sets $\{\xi_j \mid j \in J\}$ 
and $X_j$ defined in \cite{EMV-orthogonal}, and to the sets $\mathcal{N}_\eta$ 
and $\mathsf{Shuffle}^\eta_{\epsilon(z)}$ defined in 
\cite{RZ-singular-characters}. 

\begin{Theorem}
\label{T:direct-sum}
Let $\vv \in \CC^\mu$ be a seed, let $\Psi = \Psi(\vv)$, let $\Omega = \Psi 
\cap \overline \Sigma$, and let $W = W(\Psi)$. Then
\begin{align*}
V(\Omega, T(\vv))
 &= \bigoplus_{z \in \Z(\vv)} \DD{}{}{}(\Omega, \vv + z).
\end{align*} 
In particular, the set $\{\DD{}{\sigma}{\vv + z} \mid z \in \Z(\vv), \sigma 
\in W^z\}$
is a basis of $V(\Omega, T(\vv))$ and $V(\Omega, T(\vv))$ is a Gelfand-Tsetlin
module over $U$ with respect to $\Gamma$.
\end{Theorem}
\begin{proof}
By definition $V(\Omega, T(\vv)) = \sum_{z \in \ZZ^{\overline \mu}} 
\DD{}{}{}(\Omega, \vv + z)$. Now by \cite{Hump-coxeter-book}*{1.12 Theorem} 
for each $z \in \ZZ^{\overline \mu}$ there exists $\sigma \in W$ such that 
$\sigma(z) \in \Z(\vv)$. Since $W$ is the stabilizer of $\vv$ it follows from 
part (\ref{i:gamma-comparision})
of Proposition \ref{P:gamma-module-structure} that $\DD{}{}{}(\Omega, \vv + 
z) = \DD{}{}{}(\Omega, \vv + \sigma(z))$. Hence $V(\Omega, T(\vv)) = 
\sum_{z \in \Z(\vv)} \DD{}{}{}(\Omega, \vv + z)$. We next show that the sum is 
direct. Notice that the space $\DD{}{}{}(\Omega, 
\vv + z)$ consists of eigenvectors of $\Gamma = \CC[X_\mu]^{S_\mu}$ with 
eigenvalue $\ev_{\vv + z}$. If there exist $z, z' \in \Z(v)$ such that 
$\gamma(\vv + z) = \gamma(\vv + z')$ for all $\gamma \in \Gamma$ then $\vv + z 
= \sigma(\vv + z')$ for some $\sigma \in S_\mu$ and, by 
Lemma \ref{L:seed-varia}(\ref{i:unique}), $z = z'$. Hence  the sum  is  
direct. The fact that the set in question is a basis follows from Proposition 
\ref{P:gamma-module-structure}(\ref{i:gamma-basis}).
\end{proof}

\paragraph
\about{Simplicity criterion}
In this paragraph $U$ denotes a standard Galois order of type $A$ over 
$(\CC[X_\mu] \# \CC^\mu)^{S_\mu}$ with generators $X^\pm_k$ for $1 \leq k 
\leq r'$. By definition
\begin{align*}
(X_{k}^\pm)^\dagger
	&=\sym_k \left(
 	\frac{f_k^\pm}{\prod_{j = 2}^{\mu_k} (x_{k,1} - x_{k,j})} t_{k,1}^{\mp 1} 
\right)
\end{align*}
for some $f_k \in \CC[X_\mu]^{H_k}$, where $H_k$ is the stabilizer of $e_{k,1}$
in $S_\mu$. Thus we have
\begin{align*}
(X_{k}^\pm)^\dagger
 &=\sum_{i=1}^{\mu_k} \left(
 \frac{f_{k,i}^\pm}{\prod_{j \neq i}^{\mu_k} (x_{k,i} - x_{k,j})} 
 t_{k,i}^{\mp 1}
\right)
\end{align*}
where $f_{k,1}^{\pm} = (1/\mu_k) f_k^\pm$ and $f_{k,j}^\pm = \sigma \cdot 
f_{k,1}^\pm$ for any $\sigma \in S_\mu$ such that $\sigma[k](1) = j$. For the 
rest of this paragraph $f_{k,i}^\pm$ will denote the polynomials appearing in 
the formulas displayed above. 

Fix a seed $\vv$, and let $\Psi = \Psi(\vv)$ and $W = W(\Psi)$. For each 
$I = \{(k,i), \ldots, (k,j)\} \subset \Sigma(\overline \mu)$ we denote by 
$S(I)$ the group of permutations of the set $I$. This is a parabolic subgroup 
of $S_\mu$ (usually called a \emph{Young subgroup}) with minimal generating 
set $\{s_t^{(k)} \mid i \leq t \leq j - 1\} \subset S_\mu$. Using this 
notation we have  $W = \prod_{I \in \II(\overline \mu, 0)} S(I)$ and 
$W_z = \prod_{I \in \II(\overline \mu, z)} S(I)$ for each $z \in \Z(\vv)$, 
which are Young subgroups of $S_\mu$.

The following lemma describes the action of $U$ on $V(\Omega, T(\vv))$ in 
terms of the basis given in Theorem \ref{T:direct-sum}. In order to 
state the lemma  we need to fix some notation which we will also use in the 
irreducibility criterion Theorem \ref{T:simplicity}. Given $z \in \Z(\vv)$ and 
$1 \leq k \leq r'$ we will denote by $\II_k(\overline \mu, z)$ the subset 
of $\II(\overline \mu, z)$ consisting of sets of the form $I = \{(k,i), 
\ldots, (k,j)\}$. We also write $\sigma^+(I) = (j \ j-1 \ \cdots \ i)^{(k)}$ 
and $\sigma^-(I) = (i \ i+1 \ \cdots \ j)^{(k)}$.
\begin{Lemma}
\label{L:formulas}
Let $\vv \in \CC^{\mu}$ be a seed and let $z \in \Z(\vv)$. For each $1 \leq k 
\leq r'$ we have 
\begin{align*}
(X^\pm_k)^\dagger
 &= \sum_{I \in \II_k(\overline \mu, z)}
 \frac{1}{|\widetilde W_{a^\mp(I)}|} \partial_{\sigma^\mp(I)}
 \left(
 \frac{f^\pm_{a^{\mp}(I)}}{\prod_{(k,j) \notin I} (x_{a^\mp(I)}-x_{k,j})}
 t_{a^{\mp}(I)}^{\mp 1}
 \right);,\\
\DD{}{\sigma}{\vv+z} \circ (X_k^\pm)^\dagger
 &= \sum_{I \in \II_k(\overline \mu, z, \sigma^\mp(I))}
 \sum_{\tau \leq \sigma \sigma^\mp(I)}
 \D{}{\tau, \sigma \sigma^\mp(I)}{\vv + z} \left( 
 \frac{f^\pm_{a^\mp(I)}}{\displaystyle
 \prod_{(k,j) \notin I} (x_{a^\mp(I)} - x_{k,j})} 
 \right) \DD{}{\tau}{\vv + z + \delta(\mp I)},
\end{align*}
where $\II_k(\overline \mu, z, \sigma^\mp(I))$ is the subset of 
$\II_k(\overline \mu, z)$ consisting of all $I$ such that $\ell(\sigma 
\sigma^\mp(I)) = \ell(\sigma) + \ell(\sigma^\mp(I))$ and $\delta(\mp I) = 
\delta^{a^\mp(I)}$. 
\end{Lemma}
\begin{proof}
Set $\widetilde \Psi = \Psi_0(\vv + z), \widetilde \Omega = \Omega \cap 
\widetilde \Psi$, and $\widetilde W = W_z$. It is immediate that 
$\{\pm e_{a^\pm(I)} \mid I \in \II_k(\overline \mu, z)\}$ is a set of 
$\widetilde \Omega$-standard representatives of $\{\pm e_{k,1}, \ldots \pm 
e_{k,\mu_k}\}/\widetilde W$. Let $\widetilde \omega_0$ be the longest word in 
$\widetilde W$, and let $\widetilde W_{(k,t)}$ be the stabilizer of $e_{k,t}$ 
in $\widetilde W$. Then $\sigma^+(I)$ is the shortest element of the left 
coclass $\widetilde \omega_0 W_{(k,i)}$, while $\sigma^-(I)$ is the
shortest element of the left coclass $\widetilde \omega_0 W_{(k,j)}$. Thus 
using part (\ref{i:dd-form}) of Proposition \ref{P:form} we can rewrite 
$(X_k^\pm)^\dagger$ as in the statement, and the formula for 
$\DD{}{\sigma}{\vv + z} \circ (X^\pm_k)^\dagger$ is identical to the one 
obtained in Theorem \ref{T:module-structure}.
\end{proof}

Note that for each $z \in \Z(\vv)$ the longest word in $W_z$ is $\prod_{I \in 
\II(\overline \mu, z)} \omega_0(I)$, where $\omega_0(I)$ is the longest word
in $S(I)$. We will say that $z' \in \Z(\vv)$ \emph{refines} $z$ if the 
following holds: for each $J \in \II (\overline \mu, z')$ there exists $I \in 
\II(\overline \mu, z)$ such that $J \subset I$. For instance this always 
happens if $z = \vv$. If $z'$ refines $z$ then the longest element in 
$W_z^{z'}$ is equal to 
\begin{align*}
\omega_0(z,z')
	&= \prod_{I \in \II(\overline \mu, z)} \omega_0(I) 
		\prod_{J \in \II(\overline \mu, z')} \sigma^+(J).
\end{align*}
Now $\sigma \in W$ lies in $W^z$ if and only if for each $I = \{(k,i), (k,i+1),
\ldots, (k,j)\}$ in $\II(\overline \mu, z)$ we have $\sigma[k](i) < 
\sigma[k](i+1) < \cdots <\sigma[k](j)$. Hence if $z'$ refines $z$
then the longest word of $W^{z'}$ lies in $W^z$.

This observation will play a crucial role in the following 
simplicity criterion which generalizes the simplicity criterion for modules
over orthogonal Gelfand-Tsetlin algebras \cite{EMV-orthogonal}*{Theorem 11}
to modules over standard Galois orders. Note that the non-integrality 
condition in that statement is equivalent to the condition $f_{k,i}^{\pm}(\vv 
+ z) \neq 0$ below when $U$ is an orthogonal Gelfand-Tsetlin algebra.

\begin{Theorem}
\label{T:simplicity}
Let $\vv$ be a seed. If $f_{k,i}^{\pm}(\vv + z) \neq 0$ for all $z \in \Z(\vv)$ and 
all $(k,i) \in \II(\overline \mu)$ then $V(\Omega, T(\vv))$ is an 
irreducible $U$-module.
\end{Theorem}
\begin{proof}
Set $V = V(\Omega, T(\vv))$. We will show that any nonzero submodule $N 
\subset V$ is in fact equal to $V$. For each $z \in \Z(\vv)$ denote by $\pi^z: V \to \DD{}{}{} (\Omega, \vv + z)$ 
the projection to the direct summand. 
We proceed in four steps.

\noindent{\it Step 1.} If $t \in V$ and  $z \in \Z(\vv)$ are such that $\pi^z(t) \neq 0$, then 
$\DD{}{e}{\vv +z}$ is in the module $ U t$ generated by $t$. 

\noindent{\it Proof of Step 1.} 
First notice that Lemma \ref{L:sub-gt} implies $\pi^{z}(t) 
\in U t$. Now let $\m = \ker \DD{}{e}{\vv+z} \subset \Gamma$. By Theorem 
\ref{T:gamma-module} there exists a minimal $l \in \NN$ such that $\m^l 
\pi^{z}(t) = 0$, and part \ref{i:gamma-ev-kernel} of Proposition 
\ref{P:gamma-module-structure} implies that $\m^{l-1} \pi^{z}(t) = \CC 
\DD{}{e}{\vv+z'} \subset U t$. 

\noindent{\it Step 2.} 
$\DD{}{e}{\vv + z} \in N$ for all $z \in \Z(\vv)$.

\noindent{\it Proof of Step 2.}  
Step 1 implies that there exists $v' = \vv + z'$ with $z' \in \Z(\vv)$ such 
that $\DD{}{e}{v'} \in N$. To prove Step 2,  we will show that if 
$\DD{}{e}{\vv + z} \in N$ then $\DD{}{e}{\vv + z \pm \delta^{k,i}} \in N$ for 
any $(k,i) \in \II(\overline \mu)$ such that $v\pm \delta^{k,i} \in \Z(\vv)$. 
Indeed, by Lemma \ref{L:formulas} and the definition of the action of a 
co-rational Galois order on $\Gamma^*$,
\begin{align*}
\pi^{z + \delta(\mp I)} \left(X_k^\pm \cdot\DD{}{e}{\vv + z} \right)
	&= \pi^{z + \delta(\mp I)} \left(
		\DD{}{e}{\vv + z} \circ (X_k^\pm)^\dagger 
	\right) \\
 	&= \sum_{\tau \leq \sigma^\mp(I)} \D{}{\tau, \sigma^\mp(I)}{\vv + z} 
 		\left( 
 			\frac{f_{a^\mp(I)}^\pm}{\displaystyle \prod_{(k,j) \notin I} 
 				(x_{a^\mp(I)} - x_{k,j})} 
 		\right) \DD{}{\tau}{\vv + z + \delta(\mp I)}
\end{align*}
so the coefficient of $\DD{}{\sigma^\mp(I)}{\vv + \delta(\mp I)}$ is: 
\begin{align*}
\frac{f_{a^\mp(I)}^\pm(\vv + z)}{\displaystyle
	\prod_{(k,j) \notin I} (\vv_{a^\mp(I)} + z_{a^\mp(I)} - \vv_{k,j} - z_{k,j})}.
\end{align*}
Notice that this coefficient  is well defined, since by the definition of $\II(\overline 
\mu, z)$ the denominator is nonzero. Also the hypothesis on $\vv$ implies that  the 
denominator nonzero, so $\pi^{z + \delta(\mp I)} (\DD{}{e}{\vv + z} \circ 
X^\pm_k) \neq 0$, and Step 1 implies that $\DD{}{e}{\vv + z + \delta(\mp I)} 
\in N$ for all $I \in \II(\overline \mu, z)$.

\noindent{\it Step 3.}  
$\DD{}{}{}(\Omega, \vv + z) \subset N$ for all non-critical $z$, i.e. for all 
$z \in \Z(\vv)$ such that for each $I =\{(k,i), \ldots, (k,j)\} \in 
\II(\overline \mu, \vv)$ we have $z_{k,i} > z_{k,i+1} > \cdots > z_{k,j}$. 

\noindent{\it Proof of Step 3.}  
Notice that for non-critical $z$, the stabilizer of $z$ is the trivial 
subgroup of $W$ so the longest element in $W^z$ is $\omega_0$, the 
longest element of $W$. To prove Step 3,  we build a sequence $z^{(0)}, 
z^{(1)}, \ldots$ of elements in $\Z(\vv)$ as follows. First set $z^{(0)} \in 
\ZZ^{\overline \mu}$ such that $z^{(0)}_{k,i} = \min\{z_{l,j} \mid (l,j) \in 
\II(\overline \mu)\}$. Now suppose $z^{(s)}$ has been defined, and consider 
the set $L_s = \{(l,j) \in \II(\overline \mu) \mid z^{(s)}_{l,j} < z_{l,j}\}$. 
If $L_z = \emptyset$ then $z^{(s)} = z$ and we set $z^{(s+1)} = z$, otherwise 
we take $(k_s, i_s)$ to be the minimal element in $L_s$ with respect to the 
lexicographic order and set $z^{(s+1)} = z^{(s)} + \delta^{k_s, i_s}$. Clearly 
$z^{(s)} = z$ for $s \gg 0$. 

We prove by induction that $\DD{}{\omega_0^{(s)}}{\vv + z^{(s)}} \in N$, where 
$\omega_0^{(s)}$ is the longest element in $W^{z^{(s)}}$.  If $s = 0$ then by 
definition $z^{(0)}$ is a seed, and hence $\omega_0^{(0)} = e$. Since we 
already know that $\DD{}{e}{\vv + z^{(0)}} \in N$ the base case 
of the induction follows. Now take $s \geq 0$ and set $y = z^{(s)}, y' = 
z^{(s+1)}$ and $(k,i) = (k_s, i_s)$ so $y' = y + \delta^{k,i}$. The definition
of $(k,i)$ implies that there exists $j \leq \mu_k$ such that $I = \{(k,i),
\ldots, (k,j)\} \in \II(\overline \mu, y)$, and also that $\{(k,i)\} \in 
\II(\overline \mu, y')$. It follows from the characterization of the longest 
word in $W_z$ that $\omega_0^{(s+1)} = \omega_0^{(s)} \sigma^+(I)$. A 
simple computation shows that $\ell(\omega_0^{(s+1)}) = \ell(\omega_0^{(s)})
+ \ell(\sigma^+(I))$ so using Lemma \ref{L:formulas} as in the previous step
and the fact that $\DD{}{\sigma, \sigma}{v} = \ev_v$ for all $\sigma \in W$, 
we see that the coefficient of $\DD{}{\omega_0^{(s+1)}}{\vv + y'}$ in 
$\DD{}{\omega_0^{(s)}}{\vv + y} \circ (X_k^-)^\dagger$ is
\begin{align*}
\frac{f_{a^\mp(I)}^\pm(\vv + y)}{\displaystyle
	\prod_{(k,j) \notin I} (\vv_{a^\mp(I)} + y_{a^\mp(I)} 
	      - \vv_{k,j} - y_{k,j})}.
\end{align*}
and the hypothesis implies that this expression is nonzero. Hence by part 
(\ref{i:gamma-cyclic}) of Proposition \ref{P:gamma-module-structure} 
$\DD{}{}{}(\Omega, \vv + y') \subset N$, and in particular 
$\DD{}{\omega_0^{(s+1)}}{\vv + y'} \in N$.

\noindent{\it Step 4.}  
$\DD{}{}{}(\Omega, \vv + z) \subset N$ for arbitrary $z$. In particular, $N = 
V(\Omega, T(\vv))$.

\noindent{\it Proof of Step 4.}  
Fix  $z \in \Z(\vv)$. Then there exists  non-critical $y^{(0)} \in 
\Z(\vv)$ such that $y^{(0)}_{k,i} \geq z_{k,i}$ for all $(k,i)
\in \II(\overline \mu)$. For each $s \geq 0$ set $y^{(s+1)}$ to be 
$y^{(s)} - \delta^{(k_s, i_s)}$, where $(k_s, i_s)$ is the maximal element
in the set $\{(k,i) \in \II(\overline \mu) \mid y^{(s)}_{k,i} > z_{k,i}\}$
with respect to the lexicographic order; and if this set is empty then we set 
$y^{(s+1)} = y^{(s)} = z$. We claim that $\DD{}{}{}(\Omega, \vv + y^{(s)}) 
\subset N$, and prove this by induction on $s$. Since $y^{(0)}$ is non-critical
the base case of the induction follows from Step 3. Now assume that the 
inclusion holds for some $s \geq 0$, and set $(k,i) = (k_s, i_s)$. Since 
$\II(\overline \mu, y^{(s+1)})$ is a partition of the set $\II(\overline \mu)$,
there exists $I \in (\overline \mu, y^{(s+1)})$ such that $(k,i) \in I$. 
If $I = \{(k,i)\}$ then by construction $\II(\overline \mu, y^{(s)}) = 
\II(\overline \mu, y^{(s+1)})$; otherwise we have $\II(\overline \mu, y^{(s)}) = \II(\overline \mu, y^{(s+1)}) \cup \{I',
(k,i)\}\setminus \{I\}$ for $I' = I \setminus \{(k,i)\}$. From this it follows that for each $I' \in 
\II(\overline \mu, y^{(s)})$ there exists $J \in \II(\overline \mu, y^{(s)})$
such that $I' \subset J$. Thus $y^{(s)}$ refines $y^{(s+1)}$, and this implies 
that the longest element of $W^{y^{(s+1)}}$ lies in $W^{y^{(s)}}$. If we 
denote this element by $\omega_0^{(s+1)}$ then using Lemma \ref{L:formulas} 
and the hypothesis just as in the previous step we see that the coefficient of 
$\DD{}{\omega_0^{(s+1)}}{\vv + y^{(s+1)}}$ in $\pi^{\vv + 
y^{(s+1)}}\left(\DD{}{\omega_0^{(s+1)}}{\vv + y^{(s)}} \circ (X_k^{+})^\dagger 
\right)$ is not zero. Once again part 
(\ref{i:gamma-cyclic}) of Proposition \ref{P:gamma-module-structure} implies 
$\DD{}{}{}(\Omega, \vv + y^{(s+1)}) \subset N$, and since $y^{(s')} = z$ for 
$s' \gg 0$, Step 4 is proven.
\end{proof}

\end{document}